\documentclass[11pt]{article}
\usepackage[utf8]{inputenc}
\usepackage[T1]{fontenc}
\usepackage{
    amsmath,
    amssymb,
    amsthm,
    dsfont,
    graphicx,
    float,
    mdframed,
    geometry,
    mathtools,
    xcolor,
    comment,
    subcaption,
    array,
    verbatim,
    url,
    pifont,
    etoolbox
}
\usepackage[font=small]{caption}

\usepackage[
    colorlinks=true,
    citecolor=black,
    filecolor=black,
    linkcolor=black,
    urlcolor=black
]{hyperref}

\usepackage{enumitem}
\setlist[enumerate]{
    label=(\arabic*.),
    itemsep=0pt,
}
\makeatletter
\def\@tocline#1#2#3#4#5#6#7{\relax
  \ifnum #1>\c@tocdepth % then omit
  \else
    \par \addpenalty\@secpenalty\addvspace{#2}%
    \begingroup \hyphenpenalty\@M
    \@ifempty{#4}{%
      \@tempdima\csname r@tocindent\number#1\endcsname\relax
    }{%
      \@tempdima#4\relax
    }%
    \parindent\z@ \leftskip#3\relax \advance\leftskip\@tempdima\relax
    \rightskip\@pnumwidth plus4em \parfillskip-\@pnumwidth
    #5\leavevmode\hskip-\@tempdima
      \ifcase #1
       \or\or \hskip 1em \or \hskip 2em \else \hskip 3em \fi%
      #6\nobreak\relax
    \dotfill\hbox to\@pnumwidth{\@tocpagenum{#7}}\par
    \nobreak
    \endgroup
  \fi}
\makeatother
\usepackage[toc,page]{appendix}
\appto\appendix{\addtocontents{toc}{\protect\setcounter{tocdepth}{1}}}
\appto\listoffigures{\addtocontents{lof}{\protect\setcounter{tocdepth}{1}}}
\appto\listoftables{\addtocontents{lot}{\protect\setcounter{tocdepth}{1}}}

\renewcommand{\Re}{\operatorname{Re}\,}
\renewcommand{\Im}{\operatorname{Im}\,}
\newcommand{\R}{\mathbb{R}}

\numberwithin{equation}{section}

\newtheorem{theorem}{Theorem}[section]
\newtheorem{corollary}{Corollary}[theorem]
\newtheorem{lemma}[theorem]{Lemma}
\newtheorem{proposition}[theorem]{Proposition}
\theoremstyle{definition}
\newtheorem{definition}[theorem]{Definition}

\newtheorem{remark}[theorem]{Remark}

\title{Dyson Brownian motion on a Jordan curve}
\author{ 
Vladislav Guskov\footnote{KTH Royal Institute of Technology, email: vguskov@kth.se.}, 
Mingchang Liu\footnote{Capital Normal University, email: liumc\textunderscore prob@163.com.}, 
Fredrik Viklund\footnote{KTH Royal Institute of Technology, email: frejo@kth.se.}
}
\usepackage[
    backend=biber,
    style=ieee,
    sorting=nyt
]{biblatex} 
\DeclareSourcemap{
  \maps[datatype=bibtex]{
    \map{
      \step[fieldsource=options, match=\regexp{usejournal=false}, final]
      \step[fieldset=journaltitle, null]
      \step[fieldset=journal, null]
    }
  }
}
\renewbibmacro*{in:}{%
  \ifentrytype{article}
    {\iffieldundef{journaltitle}
       {}
       {\printtext{\bibstring{in}\intitlepunct}}}
    {\printtext{\bibstring{in}\intitlepunct}}%
}

\begin{document}
\maketitle
\begin{abstract}
Zabrodin recently proposed a generalization of Dyson Brownian motion to a setting where the particles are confined to a smooth Jordan curve in the plane. In this paper, we discuss a rigorous construction of such a process on a rectifiable Jordan curve and study some of its basic properties. Under further smoothness assumptions, we derive the associated Fokker-Planck-Kolmogorov equation, prove convergence towards the stationary Coulomb gas distribution, study large deviations at low temperature, and derive the limiting mean-field McKean--Vlasov equation in the many-particle limit. 
\end{abstract}

\tableofcontents

%%%%%%%%%%%%%%%%%%%%%%%%%%%%%%%%%%%%%%%%%%%
% BODY
%%%%%%%%%%%%%%%%%%%%%%%%%%%%%%%%%%%%%%%%%%%
\section{Introduction}
\subsection{Background}In an important 1962 paper \cite{dyson1962brownian}, Dyson introduced a model of ``Brownian motion gas'' now known as Dyson Brownian motion. Roughly speaking, this is a stochastic process that describes $N < \infty$ repelling Brownian particles living either on the unit circle or on the real line, interacting pairwise via the logarithmic potential,
 \begin{align} \label{eq:Coulomb_potential}
    - \beta\sum\limits_{1\le i<j\le N}\log|z_{i}-z_{j}|, \quad \beta>0.
     \end{align}
     Here $\beta$ is the inverse temperature parameter. Dyson's process has received considerable attention in both physics and mathematics, in part because of links to random matrix theory. For example, on the real line it describes the dynamics of eigenvalues of a Hermitian random matrix whose entries evolve as independent Brownian motions. See, e.g., \cite{cepa1997diffusing, cepa2001brownian} and the references therein. 

Recently, Zabrodin \cite{zabrodin2023dyson} proposed a generalization of Dyson Brownian motion to a setting where the particles are confined to a general smooth Jordan curve $\Gamma$ in the complex plane $\mathbb{C}$. In this generality, one does not expect a direct link to random matrix dynamics but it was predicted in \cite{zabrodin2023dyson} that the stationary law of the process equals the density of a planar Coulomb gas confined to the curve. That is, the probability measure on $\Gamma^{N} = \Gamma\times\dots\times\Gamma$ given by the density
\begin{align}\label{def:coulomb-gas-density}
    \frac{1}{Z_{\beta,N}(\gamma) N!}\prod_{1 \le i \neq j \le N} |z_i -z_j|^{\beta/2},\qquad z_{i}\in\Gamma, \qquad \beta > 0,
\end{align}
were $Z_{\beta,N}(\gamma)$ is the Coulomb gas partition function,
\begin{align}\label{def:partition-function}
    Z_{\beta,N}(\gamma) = \frac{1}{N!}\int_{\Gamma^{N}}\prod_{1 \le i \neq j \le N} |z_i -z_j|^{\beta/2}\prod_{i=1}^{N}|dz_{i}|.
\end{align}
 This fact, which we prove in Theorem~\ref{thm::exp_conv_to_stationary_distribution}, generalizes a similar property for classical Dyson Brownian motion on the unit circle and could, in a sense, be taken as a defining property of the process. 
 %We describe the approach of \cite{zabrodin2023dyson} below.  

The purpose of this paper is to discuss a rigorous construction of Zabrodin's process on a Jordan curve satisfying minimal regularity assumptions and to study some of its basic properties. In addition to being a rather natural object of (we feel) intrinsic interest, our motivation comes from recent work \cite{WiegmannZabrodin2022_DysonGasCurvedContour, Kurtsego, johansson2023coulomb, courteaut2025planar} on confined planar Coulomb gases and the distribution of Fekete points; the model we study here is the natural dynamical version of a Coulomb gas on the curve. We will treat Dyson Brownian motion on a Jordan arc in a forthcoming paper \cite{jordan-arc-paper}.
\subsection{Main results}
\subsubsection*{Existence}Let $\Gamma$ be a rectifiable Jordan curve. We shall always assume that $\Gamma$ is a curve in $\mathbb{C}$. By Rademacher's theorem, $\Gamma$ admits a parametrization by arc-length and we write $(\gamma(s))_{s\in[0,l)}$ for this parametrization, where $l=l(\Gamma)$ is the length of $\Gamma$. It will be convenient to extend the domain of $\gamma$ to $\mathbb{R}/ l(\Gamma)\mathbb{Z}$, making $\gamma$ an $l$-periodic function on the real line. 
Let
\begin{align}\label{def:rho}
     \rho_{\beta,N}(x)= \frac{1}{Z_{\beta, N}(\gamma)}\prod\limits_{i\neq j}|\gamma(x_{i})-\gamma(x_{j})|^{\beta/2}, \quad x_j\in \mathbb{R}, \quad \beta>0,
\end{align}
where $Z_{\beta,N}$ is as in \eqref{def:partition-function} and we are suppressing the additional condition that $1 \le i,j \le N$ in the notation.

\begin{definition}[Parametrization process]\label{def::parametrization_process}\label{def_X_process}
        Let $\Gamma$ be a rectifiable Jordan curve and fix the inverse temperature parameter $\beta > 0$. The parametrization process for Dyson Brownian motion on $\Gamma$ is a continuous strong Markov process $X = (X_{1}(t),\dots, X_{N}(t))_{t\ge 0}$ in $\mathbb{R}^N$, with diffusion matrix equal to the identity and drift given by the weak gradient of $\frac{1}{2}\log\rho_{\beta, N}$, where $\rho_{\beta, N}$ is as in \eqref{def:rho}, and such that for all $t\ge 0$, almost surely,
        \[
             X_{1}(t)\le X_{2}(t)\le \dots \le  X_{N}(t)\le X_{1}(t)+l.
        \]
\end{definition}
\begin{remark}
    Let 
    \begin{align}\label{def:weyl}
     D = \{x\in\mathbb{R}^{N}: x_{1}<x_{2}<\dots<x_{N}<x_{1}+l\}.
    \end{align}
    The ordering condition in Definition \ref{def::parametrization_process} can be stated as: the $\mathbb{R}^{N}$-valued process $X$ never leaves the closure of $D$.
    % This set is the state space for the parametrization process when $\beta \ge 1$.
    \end{remark}
    \begin{remark}
    We sometimes write $X=X^x$ to emphasize the starting point $x \in D$.
    \end{remark}

    Given the parametrization process, we define Dyson Brownian motion on $\Gamma$ by transplanting using the arc-length parametrization.
    \begin{definition}[Dyson Brownian motion on a Jordan curve]\label{def_Dyson_BM_contour} Let $\Gamma$ be a rectifiable Jordan curve. Dyson Brownian motion on $\Gamma$ is the continuous strong Markov process $\boldsymbol{Z}=(\gamma(X_{1}),\dots,\gamma(X_{N}))$, taking values in $\Gamma^{N}$, where $X=(X_{1},\dots,X_{N})$ is the parametrization process of Definition~\ref{def_X_process}. 
    \end{definition}
It is not obvious a priori that a process as in Definition~\ref{def_Dyson_BM_contour} exists. The following result gives existence in the non-colliding regime.
\begin{theorem}[Existence for $\beta \ge 1$]\label{thm:existence-intro}
    Let $\Gamma$ be a rectifiable Jordan curve. For any $\beta\ge 1$ and any collection $\boldsymbol{z} =\{z_{1},\dots, z_{N}\}$ of points on $\Gamma$, ordered counterclockwise, such that $z_{i}\neq z_{j}$, there exists a parametrization process $X = (X_{1}(t), \dots,X_{N}(t))_{t\ge 0}$, with $z_{i} = \gamma(X_{i}(0))$, for Dyson Brownian motion on $\Gamma$. 
    
    This process is the unique strong solution of the stochastic differential equation (SDE)
    \[
        dX(t) = dB(t) + \frac{1}{2}\nabla\log\rho_{\beta,N}(X(t))dt,
    \]
    where $B = (B_1(t), \ldots, B_N(t))_{t \ge 0}$ is $N$-dimensional standard Brownian motion. Moreover, for all $t\ge 0$, almost surely,
    \[
        X_{1}(t)<X_{2}(t)<\dots < X_{N}(t)<X_{1}(t)+l.
    \]
    
    Setting
    \[
        \boldsymbol{Z} = (\gamma(X_{1}(t)), \dots,\gamma(X_{N}(t)))_{t\ge 0}
    \]
    defines Dyson Brownian motion on $\Gamma$.
\end{theorem}
\noindent Theorem~\ref{thm:existence-intro} is proved in Section~\ref{sec::existence}. 
\begin{remark}
    The choice $\gamma(s)=e^{is}$ reproduces the angle process $\theta = (\theta_{1}(t),\dots,\theta_{N}(t))_{t\ge 0}$ of classical Dyson Brownian motion on the unit circle:
    \[
        d\theta_{i}(t) 
        = dB_{i}(t) + \frac{\beta}{2}\sum\limits_{j:j\neq i}\frac{1}{2}\cot\left(\frac{\theta_{i}(t)-\theta_{j}(t)}{2}\right)dt.
    \]
    The choice $\gamma(s)=s$ reproduces classical Dyson Brownian motion on the real line:
    \[
        dX_{i}(t) = dB_{i}(t) + \frac{\beta}{2}\sum\limits_{j:j\neq i}\frac{1}{X_{i}(t)-X_{j}(t)}dt.
    \]
  \end{remark}
\begin{remark}
   If $\gamma \in C^2$, we can apply It\^{o}'s formula to $Z(t) = \gamma(X(t))$, to see that our definition agrees with Zabrodin's process in \cite{zabrodin2023dyson}, see \eqref{eq: SDE_Zabrodin} below.
    \end{remark}
\begin{remark}\label{time-change-remark}
  Applying a deterministic time-change $t\to \beta t/2$ and using the scaling property of Brownian motion, one may equivalently consider the SDE
    \[
        dX(t) = \sqrt{\kappa}dB(t) - \nabla V(X(t))dt,
    \]
    where \begin{align}\label{def:V}
        V(x) = - \sum_{i<j}\log|\gamma(x_{i})-\gamma(x_{j})|,
    \end{align}
    in order to construct the parametrization process. The inverse temperature $\beta$ and the diffusion parameter $\kappa$ are related via
    \[
        \kappa = \frac{2}{\beta}.
    \]
    We shall use this parametrization when discussing large deviations below.
    \end{remark}
\begin{remark}
    The existence of the parametrization process in the regime $0<\beta<1$ is technically more difficult to establish as one has to take into account collisions between particles. That is, the process $X^{x}$ hits the boundary $\partial D$ (where $D$ is as in \eqref{def:weyl}) in finite time and must reflect back into the domain $D$. Under the general assumption that the curve $\Gamma$ is rectifiable, one has to solve an SDE with reflecting boundary conditions in the non-smooth domain $D$ with a measurable drift which is singular at the boundary. This is an interesting problem but we will not address it here, see  \cite{guskov2026thesis} for further discussion.
\end{remark}
\subsubsection*{FPK equation and Coulomb gas density}
We next discuss the Fokker-Planck-Kolmogorov (FPK) equation corresponding to the process constructed in Theorem~\ref{thm:existence-intro}. Set \[b_{\beta,N}=\frac{1}{2}\nabla\log\rho_{\beta, N}.\] We define the differential operator 
\begin{equation}\label{eq::operator_L}
    Lf = \frac{1}{2}\Delta f + b_{\beta,N}\cdot \nabla f,\qquad f\in C_{0}^{\infty}(D),
\end{equation}
and denote by $L^{*}$ its formal adjoint,
\begin{equation}\label{eq::adjoint_operator}
    L^{*}f = \frac{1}{2}\Delta f - \nabla\cdot\left( b_{\beta,N}f\right).
\end{equation}
\begin{theorem}[Fokker-Planck-Kolmogorov equation]\label{thm:FPK-intro}
Let $\Gamma$ be a rectifiable Jordan curve.  The transition probability function $P(x,t,dy)$ of the parametrization process $X=X^{x}$ with the state space $D$, defined in \eqref{def:weyl}, has a positive locally Hölder continuous density $p(x,t,y)$ for every $t>0$. The density satisfies the Fokker-Planck-Kolmogorov equation with reflecting boundary conditions, that is,
\begin{align*}
    &\partial_{t}p(x,t,\bullet)=L^{*}p(x,t,\bullet) \text{ weakly in }D,\\
    &n\cdot\left(-\frac{1}{2}\rho_{\beta,N}\frac{\nabla p(x,t,\bullet)}{\rho_{\beta,N}} \right) = 0 \text{ weakly on }\partial D.
\end{align*}
where $n$ is the unit inward-pointing normal vector.
\end{theorem}
\noindent Theorem~\ref{thm:FPK-intro} is proved in Section~\ref{sec::FPK_and_ergodicity}.
\begin{remark}
Under the additional smoothness assumption that $\gamma \in C^{2}$, we show in Proposition~\ref{prop::Kolmogorov_diffusion} that the parametrization process $X=X^x$ is a diffusion process in the sense of Kolmogorov, that is, 
\begin{enumerate}
    \item $\lim\limits_{h\to 0+}\frac{1}{h}\int_{B(x,\delta)^{c}}p(x,h,y)dy=0,$
    \item $\lim\limits_{h\to 0+}\frac{1}{h}\int_{B(x,\delta)}(y-x)p(x,h,y)dy = b_{\beta,N}(x),$
    \item $\lim\limits_{h\to 0+}\frac{1}{h}\int_{B(x,\delta)}((y-x)\cdot z )^2p(x,h,y)dy= |z|^2.$
\end{enumerate}  
\end{remark}

The mapping $g:D\to\Gamma^{N}$ given by
\[
    g(x) = \left(\gamma(x_{1}),\dots,\gamma(x_{N})\right)
\]
transplants the parametrization process $X$ to Dyson Brownian motion $\boldsymbol{Z}$ on the curve $\Gamma$. By Theorem~\ref{thm:existence-intro}, both $X$ and $\boldsymbol{Z}$ are continuous strong Markov processes. Let $P(x,t,dy)$ and $Q(\boldsymbol{z}, t, |d\boldsymbol{w}|)$ denote the transition probability functions of these processes correspondingly. If $f(x) = \boldsymbol{z}$, then the transition probability functions are related via 
\[
    Q(\boldsymbol{z}, t, \mathcal{V}) = P(x,t,g^{-1}(\mathcal{V})), \qquad \mathcal{V}\in\mathcal{B}_{S},
\]
where $\mathcal{B}_{S}$ is the Borel $\sigma$-algebra on $S$.

The next result concerns convergence towards the stationary distribution. This theorem motivates calling the process of Theorem~\ref{thm:existence-intro} ``Dyson Brownian motion''.
\begin{theorem}[Convergence towards Coulomb gas density]\label{thm::exp_conv_to_stationary_distribution}
Suppose $\gamma \in C^\infty$. Let $Q(\boldsymbol{z}, t, |d\boldsymbol{w}|)$ be the transition probability function of the Dyson Brownian motion started at $\boldsymbol{z} = (z_{1},\dots, z_{N})$. Then, as $t\to+\infty$, the measure $Q(\boldsymbol{z}, t, |d\boldsymbol{w}|)$ converges weakly, exponentially fast, to the unique stationary distribution on $(S, \mathcal{B}_{S})$ given by the Coulomb gas density 
\begin{equation}\label{eq:density_z}
     \varrho_{\beta,N}(\boldsymbol{z})=\frac{1}{Z_{\beta,N}(\gamma)}\prod_{i \neq j } |z_i -z_j|^{\beta/2},\qquad \boldsymbol{z}\in S.
\end{equation}
\end{theorem}
\noindent We prove this result in Section~\ref{sec::ergodicity}.
\subsubsection*{Large deviations}
In the next statement we consider the time-changed process, with parameter $\kappa = 2/\beta$, as in Remark~\ref{time-change-remark}. As $\kappa \to 0+$, equivalently $\beta \to +\infty$, the system converges to a deterministic gradient flow and we are interested in large deviations for this limit. In order to state the result, let $\boldsymbol{Z}= (Z_{1}(t),\dots,Z_{N}(t))_{t\ge 0}$ be (time-reparametrized) Dyson Brownian motion started at $\boldsymbol{z} = (z_{1},\dots,z_{N})$. The process is given by $Z_{i}(t) = \gamma(X_{i}(t))$, where the parametrization process $X = (X_{1}(t),\dots,X_{N}(t))_{t\ge 0}$ is a strong solution to the SDE
    \[
        dX(t) =\sqrt{\kappa}dB(t)  - \nabla V(X(t))dt,
    \]
    where $V(x)$ is as in \eqref{def:V}. Recall that $S = f(D) \subset \Gamma^N$. Let $C_{\boldsymbol{z}}(\mathbb{R}_+, S)$ be the space of continuous functions started at $\boldsymbol{z}\in S$. We equip this space with the topology of locally uniform convergence. 
\begin{theorem}[Large deviations as $\kappa \to 0+$]\label{thm::LDP_Z}
     Suppose $\gamma\in C^{3}$. The process $\boldsymbol{Z}$ satisfies a large deviation principle in $C_{{\boldsymbol{z}}}(\mathbb{R}_+, S)$, as $\kappa \to 0+$, with rate $1/\kappa$ and good rate function
    \[
        J(\boldsymbol{w}) 
        = \frac{1}{2}\int_{0}^{\infty}\sum\limits_{i=1}^{N}\left|\Re\left\{\left(\dot w_{i} - \sum\limits_{j:j\neq i}\frac{w_{i}-w_{j}}{|w_{i}-w_{j}|^{2}}\right)\overline{\tau(w_{i})}\right\}\right|^2dt,
    \]
     for absolutely continuous $\boldsymbol{w} = (w_{1}(t), \dots, w_{N}(t))_{t\ge 0}$, and set to $+\infty$ otherwise.
\end{theorem}
The proof of Theorem~\ref{thm::LDP_Z} follows directly from the contraction principle, see \cite{dembo2009large}, given a large deviation principle  for the parametrization process. The main result  of \cite{abuzaid2024large} applies in our setting, and gives a large deviation principle for the parametrization process restricted to any compact time interval $[0,T]$. In Theorem~\ref{thm::extended_LDP} we extend the large deviation principle to the infinite time interval $[0,+\infty)$. More precisely, we have the following.
\begin{theorem}\label{thm:LDP_parametrization_process_infinite_time}
   Suppose $\gamma \in C^3$. The time-reparametrized parametrization process $X^{x}$, started at $x\in D$, satisfies a large deviation principle in $C_{x}(\mathbb{R}_{+}, D)$, with rate $1/\kappa$ and good rate function
    \[
        I(u) = \frac{1}{2}\int_{0}^{\infty}\sum\limits_{i=1}^{N}\left|\dot u_{i}(t) - \sum\limits_{j:j\neq i}\Re\left\{\frac{\gamma'(u_{i}(t))}{\gamma(u_{i}(t))-\gamma(u_{j}(t))}\right\}\right|^2dt,
    \]
for absolutely continuous $u = (u_{1}(t),\dots, u_{N}(t))_{t\ge 0}$ and set to $+\infty$ otherwise.
\end{theorem}
%\end{remark}
\begin{remark}
The large deviations result implies that
\[
        \inf\limits_{u\in C_{x}(\mathbb{R}_{+}, D)}I(u) = 0.
\]
In particular, $u = (u_{1}(t),\dots, u_{N}(t))_{t\ge 0}$ satisfies $I(u)=0$ if it is a solution to the gradient system 
\[
        \dot u(t)  =  -\nabla V(u(t)).
\]
%$V(x) = - \sum_{i<j}\log|\gamma(x_{i})-\gamma(x_{j})|$.
This gradient flow is related to the distribution of Fekete points on the curve. See Section~\ref{sect::discussion}.
\end{remark}
\subsubsection*{Hydrodynamical limit}
In order to state our next result, we need some additional notation.  Let ${\boldsymbol{Z}}=(Z_1,\ldots, Z_N)$ be Dyson Brownian motion on $\Gamma$ as in Theorem~\ref{thm:existence-intro}. For $t\ge 0$, denote by 
\[
    \mu_t^{(N)}(dz)=\frac{1}{N}\sum_{i=1}^N\delta_{Z_i(t/N)}(dz)
\]
the empirical probability measure on $\Gamma$ which can be viewed as a measure on $\Gamma \times \mathbb{R}_{+}$. We allow the initial condition $\boldsymbol{Z}(0) = (Z_{1}(0),\dots,Z_{N}(0))$ to be chosen randomly, so that $\mu_0^{(N)}$ is a random probability measure on $\Gamma$. We shall assume that there exist some limiting distribution $\mu_0$ such that $\mu_0^{(N)} \to \mu_0$ weakly, as the number of particles $N\to +\infty$. The most natural choice is perhaps to start the process $\boldsymbol{Z}$ according to the stationary Coulomb gas distribution, so that $\mu^{(N)}_{0}$ converges in the limit to the electrostatic equilibrium measure on $\Gamma$. Let $\tau : \Gamma \to \partial \mathbb{D}$ be the tangent map, that is, $\tau(z)$ is the unit tangent vector to $\Gamma$ at the point $z$, oriented counterclockwise. The following theorem shows that the system approaches a mean-field McKean-Vlasov equation depending on the curve, in the many-particle limit, $N \to +\infty$.
\begin{theorem}[Hydrodynamical limit]\label{thm:hydro}
Suppose $\gamma \in C^2$. Let $\mu_{0}$ be a probability measure on $\Gamma$ and suppose that 
$\mu_{0}^{(N)} \to \mu_{0}$ in the weak sense, as $N\to +\infty$. Then, the family $\{(\mu_{t}^{(N)})_{t\ge 0}\}_{N\ge 1}$ is tight, and any subsequential limit $\mu= (\mu_{t})_{t\ge0}$ satisfies
\begin{align*}%\label{eqn::hydro_SDE_1}
    d\mu_{t}(f) = \frac{\beta}{4}\int_\Gamma\int_\Gamma\left(\partial_s f(z)\Re\frac{\tau(z)}{z-w}-\partial_s f(w)\Re\frac{\tau(w)}{z-w}\right)\mu_t(dz)\mu_t(dw)dt,
\end{align*}
for all test functions $f$, smooth in a neighborhood of $\Gamma$.
\end{theorem}
\noindent We prove Theorem~\ref{thm:hydro} in Section~\ref{sec::hydro}.

\subsection{Discussion}\label{sect::discussion}
\subsubsection*{Zabrodin's approach}
Let us briefly describe the main results of \cite{zabrodin2023dyson},  which inspired the present paper. Let $\gamma \in C^2$ with $\tau : \Gamma \to \partial \mathbb{D}$ the tangent map. Let $\nu = -i\tau  : \Gamma \to \partial D$ be the outward pointing unit normal map. Finally, $k : \Gamma \to \R$ is the curvature map of $\Gamma$, that is, $k(z)$ is the curvature of $\Gamma$ at the point $z$ defined as
\begin{align}\label{eqn: curvature}
k(z) = \frac{d \arg \tau(z)}{d |z|}.
\end{align}
Via these quantities Zabrodin \cite{zabrodin2023dyson} proposed to construct a process $\boldsymbol{Z}$ as a solution to the system of SDEs
\begin{align}\label{eq: SDE_Zabrodin}
    dZ_{i} = \sqrt{\kappa}\tau(Z_{i})dB_{i} + \left(\tau(Z_{i})(\partial_{s_{i}}E)(\boldsymbol{Z}) - \frac{\kappa}{2}\nu(Z_{i})k(Z_{i})\right)dt, \quad i=1,\ldots,N.
\end{align}
Here $B = (B_1, \ldots, B_N)$ is $N$-dimensional standard Brownian motion, and $E : \Gamma^N \to \R$ is the real-valued energy functional on the configuration $\boldsymbol{Z}$ of particles defined by
\begin{align}\label{eqn: E_functional}
E(\boldsymbol{z})= 2 \sum_{i < j} \log |z_i - z_j| + \sum_i W(z_i),
\end{align}
where $W : \mathcal{N}(\Gamma) \to \R$ a real-valued external potential defined in a neighborhood $\mathcal{N}(\Gamma)$ of the curve. The external potential does not appear in our setup; in that way Zabrodin's setup is slightly more general. The quantity $\partial_{s_i} E$ refers to the tangential derivative of the energy functional with respect to the $z_i$ variable. 

For $f : \mathcal{N}(\Gamma) \to \R$ defined in a neighborhood of the curve the tangential derivative is
\begin{align}\label{eq::tangential_derivative}
    (\partial_{s} f)(z) = \lim_{\substack{h \to 0 \\ h \in \R}} \frac{f(z + \tau(z) h) - f(z)}{h} = 2 \Re \left( \tau(z) \partial_{z} f(z) \right), \quad z \in \Gamma.
\end{align}
For a function $f(\boldsymbol{z})=f(z_{1},\dots,z_{N})$ on $\Gamma^N$, $\partial_{s_{i}}f$ refers to the tangential derivative in the $i^{th}$ variable, while $\partial_z f(z)$ refers to the holomorphic derivative.

For $E$ given by \eqref{eqn: E_functional} the tangential derivative is 
\[
(\partial_{s_i} E)(\boldsymbol{z}) = \Re \left\{ \sum_{j:j \neq i} \frac{\tau(z_i)}{z_i - z_j} \right\} + 2 \Re \left\{\tau(z_i) (\partial W)(z_i) \right\}.
\]
Given this setup, \cite{zabrodin2023dyson} proposes to define the \emph{Dyson diffusion process} $\boldsymbol{Z} = (Z_1, \ldots, Z_N)$ on a smooth Jordan curve $\Gamma$ as a solution to the system of SDEs \eqref{eq: SDE_Zabrodin}. Next if $\kappa = 2/\beta$, then, the function $u : \Gamma^N \to [0, \infty)$ defined by
\[
u(\boldsymbol{z}) = \prod_{i \neq j} |z_i - z_j|^{\beta} \prod_{i} e^{\beta W(z_i)}.
\]
is claimed to be a solution to the stationary Fokker-Planck-Kolmogorov equation $L^*u = 0$, where $L^*$ is the (formal) adjoint of the generator
\begin{align*}
    L = \sum\limits_{i=1}^{N}\left(\frac{\kappa}{2}\partial_{s_{i}}^{2} + (\partial_{s_{i}}E)\partial_{s_{i}}\right).
\end{align*}
No mathematical proofs of these statements were given in \cite{zabrodin2023dyson}.

Note that the notation of \cite{zabrodin2023dyson} differs from ours: Zabrodin's $\beta$-parameter equals $1/2$-times our $\beta$-parameter. We have chosen to use the convention of \cite{johansson2023coulomb, courteaut2025planar}, for which the determinantal model corresponds to $\beta = 2$.

Let us check that our construction matches the just described one. For the arc-length parameterization $\gamma$ that traces out the curve $\Gamma$ in a counterclockwise manner we recall the basic expressions of tangent vectors and curvature for plane curves: 
\[
\tau(\gamma(t)) = \gamma'(t), \quad -k(\gamma(t)) \nu(\gamma(t)) = \gamma''(t). 
\]
Itô's formula, which can be applied since we assume $\gamma \in C^2$, shows that for a real-valued process $X_i$ we have
\[
d \gamma(X_i(t)) = \gamma'(X_i(t)) dX_i(t) + \tfrac{1}{2} \gamma''(X_i(t)) d \langle X_i, X_i \rangle(t),
\]
and by the equations above this can be rewritten in terms of $Z_i(t) = \gamma(X_i(t))$ as
\[
d Z_i(t) = \tau(Z_i)  dX_i(t) - \tfrac{1}{2} k(Z_i) \nu(Z_i) d \langle X_i, X_i \rangle(t). 
\]
In order to match \eqref{eq: SDE_Zabrodin} we see that we should take
\[
dX_i(t) = \sqrt{\kappa} \, dB_i(t) + (\partial_{s_i} E)(\boldsymbol{Z}) \, dt.
\]
Up to a linear time change of the process, and with $W \equiv 0$, this choice in turn matches \eqref{def_X_process} since
\[
\partial_{x_i} \log \rho_{\beta,N}(X) = \beta \sum_{j : j \neq i} \Re \left( \frac{\gamma'(X_i)}{\gamma(X_i) - \gamma(X_j)} \right) = \beta \sum_{j : j \neq i} \Re \left( \frac{\tau(Z_i)}{Z_i - Z_j} \right).
\]
\subsubsection*{Gradient flow and Fekete points}
For $u:[0,+\infty)\to \mathbb{R}^{N}$, consider the gradient system 
\begin{equation}\label{eq::gradient_system_V}
\begin{cases}
    \dot u(t) = - \nabla V(u(t)), \quad t>0\\
    u(0) = u_{0}\in D,
\end{cases}
\end{equation}
where $V$ is as in \eqref{def:V}.
The potential $V(x)\to +\infty$ as $x\to \partial D$. A solution to the gradient system \eqref{eq::gradient_system_V} does not hit the boundary $\partial D$ since $t\to V(u(t))$ is a non-increasing function:
\[
    \frac{d}{dt}V(u(t)) = \nabla V(u(t))\cdot \frac{du(t)}{dt}  = -|\nabla V(u(t))|^2 \le 0.
\]
In particular, it stays inside the set $\{x\in D: V(x)\le V(u_{0})\}$. Local boundedness and $l$-periodicity of $V$ along $\hat{e} = (1,\dots, 1)^{T}$ imply that a solution does not blow up in finite time. Stationary solutions satisfy
\[
\nabla V(u(t)) \equiv 0, \qquad t \ge 0.
\]
A particularly interesting class of stationary solutions are the Fekete points. Recall that the $N$:th level Fekete points of $\Gamma$ is an $N$-tuple $\{z_{1}^{*},\dots , z_{N}^{*}\}$ of points that maximizes the discriminant, equivalently minimizes the logarithmic energy,
\[
\Delta_N(\Gamma) := \prod_{k\neq \ell} |z_k^*-z_\ell^*| =  \max_{z_1, \ldots, z_N \in \Gamma} \prod_{k\neq \ell} |z_k-z_\ell|.
\]
The simplest example is when $\Gamma=\mathbb{T} = \partial \mathbb{D}$. In this case any $N$-tuple of equidistant points form Fekete points, which are consequently not unique. Moreover, $\Delta_N(\mathbb{T})=N^N$. It is a classical result that
\[
\lim_{N \to \infty} \Delta_N(\Gamma)^{1/(N(N-1))} = \textrm{cap}(\Gamma),
\]
that is, the transfinite diameter equals the logarithmic capacity of $\Gamma$. It is known in a general setting (for compact sets with infinitely many points) that that the empirical measure of Fekete points converges to the electrostatic equilibrium measure, as $N\to+\infty$. See, e.g., \cite[p. 285]{Hille1973_AFT2} and \cite{Pommerenke1967Fekete} for further information.

\section*{Acknowledgements}
V.G.\ acknowledges support from the Göran Gustafsson Foundation for Research in Natural Sciences and Medicine. M.L.\  acknowledges support from the Knut and Alice Wallenberg Foundation. F.V.\ acknowledges support from the Knut and Alice Wallenberg Foundation, the Göran Gustafsson Foundation for Research in Natural Sciences and Medicine,  and the Simons Foundation. We thank Tom Alberts and Kurt Johansson for discussions and comments on an earlier version of the paper.

\section{Existence}\label{sec::existence}
The objective of this section is to establish existence of a process as in Definition~\ref{def_Dyson_BM_contour}. We show that in the regime $\beta\ge 1$ there exists a parametrization process, constructed as a unique strong solution to a stochastic differential equation, which after transplanting it to the curve $\Gamma$ with the arc-length parametrization yields a continuous strong Markov process. That is, we prove Theorem~\ref{thm:existence-intro}. However, the existence of Dyson Brownian motion in the regime $0<\beta<1$ remains an open problem.  

\subsection{Strong solution}
First, we establish existence and uniqueness of the parametrization process as a strong solution to the system of SDEs:
\begin{align}\label{eq::SDE_parametrization_process_X}
\begin{cases}
    dX(t) = dB(t)+b(X(t))dt, \ t\ge0,\\
    X(0)=x\in D.
\end{cases}
\end{align}
Here
\[
D = \{x\in\mathbb{R}^{N}: x_{1}<x_{2}<\dots<x_{N}<x_{1}+l\},
\]
 is an open, unbounded, convex domain, with smooth boundary except at intersections of the hyperplanes which form its boundary.
$B=(B_{1},\dots,B_{N})$ is $N$-dimensional standard Brownian motion. The drift term  \[b = \frac{1}{2}\nabla\log\rho_{\beta,N}\] is given by the logarithmic gradient of the density
\begin{align*}
    \rho_{\beta,N}(x) = \frac{1}{Z_{\beta, N}(\gamma)}\prod_{i\neq j}|\gamma(x_{i})-\gamma(x_{j})|^{\beta/2}, \quad \beta>0.
\end{align*}
The gradient is understood in the weak sense. The $i$:th component of the drift is given almost everywhere by 
\begin{align*}
b_{i}(x) = \frac{\beta}{2}\sum\limits_{j:j\neq i}\Re\left\{\frac{\gamma'(x_{i})}{\gamma(x_{i})-\gamma(x_{j})}\right\}.
\end{align*}
 For a general rectifiable Jordan curve $\Gamma$ the drift $b$ is a locally bounded measurable function in $ D$, and $|b(x)|\to+\infty$ as $x\to\partial D$. Indeed $|\gamma'| = 1$ where defined.

\subsubsection{Strong local solution}
Using results from \cite{krylov2005strong} we can construct a unique strong solution up to the time when the process reaches the boundary of the domain $ D$, that is, up to the first collision time of any two particles on the curve. Denote the exit time of the domain $D$ by 
\[
    \tau_{\partial D}(X)= \inf\{t\ge0:X(t)\notin  D\}.
\]
In order to define a process for all $t\ge0$,  we can employ Alexandrov compactification of the state space $ D$. Let $*\notin D$ be some point, define its neighborhoods to be complements in $ D$ of closed bounded subsets of $ D$. Then, $ D'= D\cup\{*\}$ is a compact topological space.

Let $(\Omega, \mathcal{F}, (\mathcal{F}_{t})_{t\ge s}, \mathbb{P})$ be a complete filtered probability space satisfying the usual conditions (the filtration is right-continuous and $\mathcal{F}_{0}$ contains all null sets). Let $B = (B(t))_{t\ge 0}$ be $N$-dimensional standard Brownian motion adapted to the filtration $(\mathcal{F}_{t})_{t\ge0}$. 
    
By \cite[Theorems 2.1, 2.5]{krylov2005strong}, for any $x\in D$ and $s\ge0$, there exists a continuous $ D'$-valued process $X\equiv X^{s,x}=(X^{s,x}(t))_{t\ge s}$, with $X^{s,x}(s)=x$, such that 
\begin{enumerate}
        \item $X$ is adapted to the filtration;
        \item For any $t\ge s$ on the event $\{t\ge \tau_{\partial D}(X) \}$ $$X(t)=* \text{ a.s.;}$$
        \item For any $t\ge s$ on the event $\{t<\tau_{\partial D}(X)\}$ 
        \begin{align*}
            \int_{s}^{t}|b(X(r))|^{2}dr<\infty\quad \text{ a.s.}
        \end{align*}
        and
        \begin{align*}
            dX(t) = dB(t)+b(X(t))dt\quad \text{ a.s.;}
        \end{align*}
        \item Pathwise uniqueness holds;
        \item The law of $X$ on $C([s,\infty),  D')$ is uniquely determined by the drift $b$;
        \item Let $\mathbb{P}_{s,x}$ be the probability distribution of $X^{s,x}$ on $C([s,\infty), D')$. Define the filtration  
        \[
            \mathcal{G}_{t} = \sigma\left\{(\omega(r))_{r\in [s,t]}: \omega\in C([s,\infty),  D') \right\}.
        \]
        The process
        \[
            \mathbb{M}_{s,x}= (C([s,\infty),  D'), (\mathcal{G}_{t})_{t\ge s}, (X(t))_{t\ge s}, \mathbb{P}_{s,x})
        \]
        is a strong Markov process;
        \item For any Borel bounded function $f$ on $ D'$ and $T>0$, the function 
        \[
            (s,x)\to\mathbb{E}\left[f(X^{s,x}(T-s))\right]
        \]
        is continuous on $[0,T]\times D$.
\end{enumerate}

\subsubsection{Strong global solution}
As stated above, there exists a unique strong solution to $dX(t)=dB(t)+b(X(t))dt,\  X(0)=x\in D$, defined up to the exit time of the  domain $D$. However, the drift has the property that $|b(x)| \to +\infty$ as $x\to\partial D$, which tends to keep the process away from the boundary. In the case of the real line and the circle, for any $\beta>0$ the process never leaves $\overline{D}$, and for $\beta\ge 1$ it never hits $\partial D$, see \cite{cepa1997diffusing, cepa2001brownian}. For a general rectifiable curve $\Gamma$ we do not know of a similar result that covers the whole range of $\beta>0$. There are nonetheless sufficient criteria, imposed on the density, ensuring that the process never reaches the boundary. In this regime, a local solution becomes a global one.

\begin{remark}
Already in \cite{krylov2005strong} a sufficient condition on the drift $b$ for the process to never hit the boundary $\partial D$ was given, in the case of gradient drift $b = - \nabla\psi$. It is sufficient that the potential function $\psi$ satisfies the inequality $\Delta\psi \le h e^{\varepsilon\psi}$ for some $\varepsilon\in [0,2)$ and a continuous function $h: D\to\mathbb{R}$ that, for all $\sigma>0$, is  $L^r(D, e^{-\sigma|x|^2}dx)$-integrable for some $r=r(\sigma)>1$. Applied to our setting, this criterion gives non-collision of particles almost surely if $\gamma\in C^{2}$ and $\beta>1$ and consequently a strong global solution under these assumptions, which would cover the case considered in \cite{zabrodin2023dyson}.
\end{remark}

Let $D'=D\cup\{*\}$ be the Alexandrov compactification of $D$. 
Let $x\in D$ and $X^{x}$ be a unique strong local solution with initial condition $X^{x}(0)=x$. Denote by $\mathbb{P}_{x}$ the distribution of $X^{x}$ on 
\[
    \tilde\Omega := \{\omega \in C(\mathbb{R}_{+}, D'): \omega(t) = * \text{ for } t\ge \zeta(\omega)\},
\]
where $\zeta(\omega) = \inf\{t\ge 0: \omega(t)\notin D\}$ is the lifetime of the process.
\[
    \mathbb{M} = (\tilde\Omega ,  \mathcal{G}, (\mathcal{G}_{t})_{t\ge 0},(X(t))_{t\ge 0},  (\mathbb{P}_{x})_{x\in D})
\]
is a strong Markov process. This is \cite[Theorems 2.1 and 2.5]{krylov2005strong}. 

We want to show that for all $x\in D$ the process $X^{x}$ does not hit the boundary $\partial D$ almost surely.  In other words, for
\[
    \tilde\Omega_{0} = \{\omega\in \tilde\Omega: \omega(0)\in D,\ \zeta(\omega)=\infty\}
\]
we want to prove that
\[
    \mathbb{P}_{x}[\tilde{\Omega}_{0}]=1 \text{ for all }x\in D.
\]
Following \cite[Section 4]{albeverio2003strong}, we can find a diffusion process, via the Dirichlet form theory, that posses this property, which then can be transferred to the strong local solution as explained below.
\subsubsection*{Dirichlet form}
Let $\mathcal{E}$ be a non-negative definite, symmetric, bilinear form defined by 
\[
    \mathcal{E}(f,g):=\frac{1}{2}\sum\limits_{i=1}^{N}\int_{D}\frac{\partial f}{\partial x_{i}}\frac{\partial g}{\partial x_{i}} \rho_{\beta,N}(x)dx,\qquad f,g\in C^{\infty}_{0}(D),
\]
where $C^{\infty}_{0}(D)$ denotes the space of compactly supported smooth function in $D$. $C^{\infty}_{0}(D)$ is a dense linear subspace of $L^{2}(D,\mu)$. 

The form $\mathcal{E}$ is closable if for any sequence  $(f_{n})$ in $ C^{\infty}_{0}(D)$ such that $\lim_{n\to\infty}||f_{n}||_{L^{2}(D,\mu)}$ and $\lim_{n,m\to\infty}\mathcal{E}(f_{n}-f_{m}, f_{n}-f_{m})=0$ it follows that $\lim_{n\to\infty}\mathcal{E}(f_{n},f_{n})=0$.
\begin{lemma}
    The bilinear form $(\mathcal{E}, C^{\infty}_{0}(D))$ closable on $L^{2}(D,\mu)$.
\end{lemma}
\begin{proof}
    Since $D= \{\rho_{\beta,N}>0\}$ and $x\to \rho_{\beta,N}(x)$ in continuous, the lemma follows from the derivation in \cite[Section II.2.a]{ma1992introduction}.
\end{proof}

For the form $\mathcal{E}$ to be closable is a necessary and sufficient condition to have a closed extension, see \cite[Chapter]{fukushima2011dirichlet}. Denote by $(\mathcal{E}, \mathcal{D}(\mathcal{E}))$ the closure of $(\mathcal{E}, C^{\infty}_{0}(D))$ on $L^{2}(D,\mu)$, i.e., $\mathcal{D}(\mathcal{E})$ is a Hilbert space with respect to the inner product $\mathcal{E}(f,g) + (f,g)_{L^{2}(D,\mu)}$.
\begin{lemma}
    $(\mathcal{E}, \mathcal{D}(\mathcal{E}))$ is symmetric, strongly local, regular Dirichlet form on $L^{2}(D,\mu)$.
\end{lemma}
\begin{proof}
    See the proof of \cite[Theorem 4.1.4]{baur2014elliptic} or
    \cite[Proposition 4.2]{baur2013construction} and references therein.
\end{proof}

\subsubsection*{Capacity}
The capacity of a set with respect to the form $(\mathcal{E}, \mathcal{D}(\mathcal{E}))$ is defined, for an open set $A\subset D$, by
\[
    \text{Cap}_{\mathcal{E}}(A):= \inf\{\mathcal{E}(f,f)+(f,f)_{L^{2}}: f\in\mathcal{D}(\mathcal{E}),\ f\ge 1 \ \mu\text{-a.e.}\},
\]
and, for any set $B\subset D$, by
\[
    \text{Cap}_{\mathcal{E}}(B):= \inf\{\text{Cap}_{\mathcal{E}}(A): A \text{ is open and }  B\subset A\}.
\]
\begin{proposition}\label{prop:capacity_zero}
    For $\beta \ge 1$, $\text{Cap}_{\mathcal{E}}(\{\rho_{\beta,N}(x)=0\})=0$.
\end{proposition}
\begin{proof}
The density $\rho_{\beta,N}(x)$ is bounded by 
\[
    \rho_{\beta,N}(x) = \frac{1}{Z_{\beta,N}(\gamma)}\prod\limits_{i\neq j}|\gamma(x_{i})-\gamma(x_{j})|^{\beta/2} 
    \le  \frac{1}{Z_{\beta,N}(\gamma)}\prod\limits_{i<j}h(|x_{i}-x_{j}|),
\]
where $h(r) := \min(r, l-r)^{\beta}$, $r\in [0,l]$. The function $h$ satisfies the following integrability conditions: for any $\beta\ge 1$,
\[
    \int_{0+}\frac{1}{h(r)}dr=+\infty\quad  \text{ and }\quad \int^{l-}\frac{1}{h(r)}dr=+\infty.
\]
This integrability condition on the density forces the capacity (with respect to the Dirichlet form) of the set $\{\rho_{\beta,N}=0\}$ to be zero. The derivation of this claim follows, with minimal changes, the proof of \cite[Proposition 2.1]{inukai2006collision}.
\end{proof}

Recall 
\[
\begin{split}
    &\tilde\Omega := \{\omega \in C(\mathbb{R}_{+}, D'): \omega(t) = * \text{ for } t\ge \zeta(\omega)\},\\
    &\tilde\Omega_{0} := \{\omega\in \tilde\Omega: \omega(0)\in D,\ \zeta(\omega)=\infty\},
\end{split}
\]
where $\zeta(\omega) = \inf\{t\ge 0: \omega(t)\notin D\}$ is the lifetime of the process.
The theory of regular Dirichlet forms provides an existence of a diffusion process, associated with $(\mathcal{E}, \mathcal{D}(\mathcal{E}))$,
\[
    \mathbb{M}_{x} = \{\tilde\Omega, \tilde{\mathcal{G}},(\tilde{\mathcal{G}}_{t})_{t\ge 0}, (X(t))_{t\ge0}, \tilde{\mathbb{P}}_{x}\}
\]
on $D'$ for quasi-every starting point $x\in D$, i.e., points in $D$ outside of an exceptional set. The exceptional set is fairly small --- it has zero capacity with respect to the Dirichlet form --- but is usually not explicit.  See \cite{fukushima2011dirichlet, ma1992introduction} for more on Dirichlet form techniques.
\begin{remark}
    The theory of regular Dirichlet forms provides distributions $\mathbb{P}_{x}$ for starting points $x\in D$ outside of an exceptional set. The exceptional set is fairly small --- it has zero capacity with respect to the Dirichlet form (see below) --- but is usually not explicit. However, the measure $\mu(dx)=\rho_{\beta,N}(x)dx$ does not charge sets of zero capacity. This allows the process associated with $(\mathcal{E}, \mathcal{D}(\mathcal{E}))$ to start from a random point 
\end{remark}
For $\mu(dx) = \rho_{\beta,N}(x)dx$, define 
\[
    \mathbb{P}_{\mu}:= \int_{D}\mathbb{P}_{x}\ \mu(dx) \text{ and }  \tilde{\mathbb{P}}_{\mu} := \int_{D}\tilde{\mathbb{P}}_{x}\ \mu(dx).
\]
By Proposition~\ref{prop:capacity_zero} and \cite{ma1992introduction},
\[
    \tilde{\mathbb{P}}_{\mu}[\tilde\Omega_{0}]=1.
\]
\begin{lemma}
\[
        \tilde{\mathbb{P}}_{\mu} = \mathbb{P}_{\mu}.
\]
\end{lemma}
\begin{proof}
The proof follows \cite[Lemma 4.2]{albeverio2003strong}. Let $(P_{t})_{t\ge0}$ be the transition semigroup of $\mathbb{M}$. Since $\tilde{\mathbb{M}}$ is associated to the Dirichlet form $(\mathcal{E}, \mathcal{D}(\mathcal{E}))$, its transition semigroup $(\tilde P_{t})_{t\ge 0}$ is related to $(P_{t})_{t\ge0}$ by: $\tilde P_{t}f$ is a $\mu$-version of $P_{t}f$ for $f\in L^{2}(D,\mu)$, see \cite[Remark 2.5 (ii)]{ma1992introduction}.

Let $A_{i}$, $i=0,\dots,n$, be Borel-measurable set in $D$. Let $0<s_{1}<\dots <s_{n}$.
By Markov property, cf. \cite[Lemma 4.1.2]{fukushima2011dirichlet}, we have
\begin{align*}
    &\tilde{\mathbb{P}}_{\mu}[X(0)\in A_{0},X(s_{1})\in A_{1}, \dots X(s_{n})\in A_{n}]\\
    &= \int_{D}\tilde{\mathbb{P}}_{x}[X(0)\in A_{0},X(s_{1})\in A_{1}, \dots X(s_{n})\in A_{n}]\mu(dx)\\
    &=\int_{D}\mathds{1}_{A_{0}}(x) (\tilde P_{s_{1}}\mathds{1}_{A_{1}}\dots \mathds{1}_{A_{n-1}}( \tilde P_{s_{n}-s_{n-1}}\mathds{1}_{A_{n}}))(x)\mu(dx)\\
    &=\int_{D}\mathds{1}_{A_{0}}(x) ( P_{s_{1}}\mathds{1}_{A_{1}}\dots \mathds{1}_{A_{n-1}}(  P_{s_{n}-s_{n-1}}\mathds{1}_{A_{n}}))(x)\mu(dx)\\
    &={\mathbb{P}}_{\mu}[X(0)\in A_{0},X(s_{1})\in A_{1}, \dots X(s_{n})\in A_{n}].\qedhere
\end{align*}
\end{proof}
\begin{lemma}
    $\mathbb{P}_{x}[\tilde\Omega_{0}]=1$ for all $x\in D = \{\rho>0\}$.
\end{lemma}
\begin{proof}
Since ${\mathbb{P}}_{\mu}[\tilde\Omega\setminus \tilde\Omega_{0}] = \int_{D}{\mathbb{P}}_{x}[\tilde\Omega\setminus \tilde\Omega_{0}] \mu(dx)=0$ and $\mu(dx)=\rho_{\beta,N}(x)dx$ with $\rho_{\beta,N}>0$ on $D$, it holds that ${\mathbb{P}}_{x}[\tilde\Omega\setminus \tilde\Omega_{0}]=0$ almost everywhere on $D$. We show that it holds everywhere.

Let $\theta_{s}:\tilde\Omega\to \tilde\Omega$ denote the translation operator $\theta_{s}(\omega) := (\omega(t+s))_{t\ge0}$. Then, $\theta^{-1}_{s}\tilde\Omega=\tilde\Omega$ and $\theta^{-1}_{s}\tilde\Omega_{0}=\tilde\Omega_{0}$. Fix $x\in D$. By the telescopic property of conditional expectation, 
\[
\begin{split}
    \mathbb{P}_{x}(\tilde\Omega \setminus \tilde\Omega_{0})
    &=  \mathbb{P}_{x}(\theta^{-1}_{s}(\tilde\Omega \setminus \tilde\Omega_{0}))\\
    &=\int_{\tilde\Omega}\mathds{1}_{\{\tilde\Omega \setminus \tilde\Omega_{0}\}}(\theta_{s}\omega)\mathbb{P}_{x}(d\omega)\\
    &= \int_{\tilde\Omega}\mathbb{E}_{x}\left(\mathds{1}_{\{\tilde\Omega \setminus \tilde\Omega_{0}\}}(\theta_{s}\tilde\omega)\Big|\mathcal{G}_{s}\right)(\omega)\mathbb{P}_{x}(d\omega).
\end{split}
\]
By the Markov property,
\[
    \mathbb{E}_{x}\left(\mathds{1}_{\{\tilde\Omega \setminus \tilde\Omega_{0}\}}(\theta_{s}\tilde\omega)\Big|\mathcal{G}_{s}\right)(\omega) 
    = \mathbb{P}_{\omega(s)}\left(\tilde\Omega\setminus\tilde\Omega_{0}\right),
\]
where $\omega(0)=x$. The transition probability function of $X^{x}$ is absolutely continuous with respect to the Lebesgue measure; we prove this fact in Proposition~\ref{prop::FPK_X_process}. This leads to the identity
\[
\begin{split}
    \mathbb{P}_{x}(\tilde\Omega \setminus \tilde\Omega_{0}) 
    &= \int_{\tilde\Omega}\mathbb{P}_{\omega(s)}(\tilde\Omega\setminus\tilde\Omega_{0})\mathbb{P}_{x}(d\omega)\\
    &=\mathbb{E}\left[\mathbb{P}_{X^{x}(s)}(\tilde\Omega\setminus\tilde\Omega_{0})\right]\\
    &= \int_{D}\mathbb{P}_{y}(\tilde\Omega\setminus\tilde\Omega_{0})P(x,s,dy)\\
    & = \int_{D}\mathbb{P}_{y}(\tilde\Omega\setminus\tilde\Omega_{0})p(x,s,y)dy\\
    & = 0.
\end{split}
\]
Since $x\in D$ was an arbitrary fix point, it follows that $\mathbb{P}_{x}[\tilde\Omega_{0}]=1$ for every $x\in D$.
\end{proof}
\noindent This completes the proof of existence of a unique strong solution to the SDE
\[
    dX(t) = dB(t) + \frac{1}{2}\nabla \log\rho_{\beta,N}(X(t))dt,
\]
for any initial point $x\in D$, in the regime $\beta\ge 1$.

\subsection{Back to the curve}\label{sec::back_to_the_curve}
Let $(\Omega, \mathcal{F}, (\mathcal{F}_{t})_{t\ge 0}, \mathbb{P})$ be a complete filtered probability space on which the parametrization process  $X^{s,x}=(X^{s,x}(t))_{t\ge s}$ is defined, with the starting point $X^{s,x}_{s}=x\in D$. Denote by $\mathcal{B}_{D}$ the Borel $\sigma$-algebra on $D$. $X^{s,x}$ is a continuous  $D$-valued time-homogeneous strong Markov process with the transition probability functions
\begin{align*}
    P(s,x,t,B) = \mathbb{P}\left[X^{s,x}(t)\in B\right],\  t\ge s,\ B\in\mathcal{B}_{D},        
\end{align*}
which are determined by 
\begin{align*}
    P(x,t,B) = P(0, x,t,B), \ t\ge 0, \ B\in\mathcal{B}_{D}.
\end{align*}
Denote by $X^{x} =X^{0,x}$. Then, $(X^{x}, \mathbb{P})$ is a time-homogeneous strong Markov process in the state space $(D, \mathcal{B}_{D})$, where 
\[
    D = \left\{x\in\mathbb{R}^{N}:x_{1}<\dots<x_{N}<x_{1}+l\right\},
\]
with the transition probability function $P(x,t,B)$.
\subsubsection*{Quotient parametrization process}
The parametrization process is defined on an unbounded domain $D$. In this section, we redefine the process on a compact "cylinder". Although this middle step is not required to transplant the process to the curve, we will make use of this construction in Section~\ref{sec::ergodicity}.

Let $f:D\to \mathbb{R}^{N}$ be a mapping defined by 
\begin{equation}\label{eq::function_f_D_to_E}
    f(x) = x - (x_{1} - (|x_{1}| \text{ mod } l))\hat{e},
\end{equation}
where $\hat{e} = (1,\dots, 1)^{T}$. The image of $D$ under $f$ is
\[
    E  = f(D) = \{x\in\mathbb{R}^{N}: 0\le x_{1} < \dots< x_{N}< x_{1}+l<2l\}.
\]
One can identify the points on two parts of the boundary $\partial E\cap \{x_{1}=0\}$ and $\partial E\cap \{x_{1}=l\}$ by
\[
    \begin{pmatrix}
    0 \\
    x_{2} \\
    \dots\\
    x_{N}
    \end{pmatrix}
    \sim
    \begin{pmatrix}
    l \\
    x_{2}+l \\
    \dots\\
    x_{N}+l
    \end{pmatrix},
\]
making the space $E$, equipped with the quotient topology, into a ``cylinder''. Then, $f:D\to E$ is a continuous surjective mapping.

For $k\in\mathbb{Z}$, the sets $\{E + kl\hat{e}\} = \{x+kl\hat{e}: x\in E\}$ form a decomposition of $D$:
\[
    D = \bigcup\limits_{k=-\infty}^{\infty}\{E + kl\hat{e}\}.
\]
For a subset $V\subset E$, we have
\[
    f^{-1}(V) = \bigcup\limits_{k=-\infty}^{\infty}\{V + kl\hat{e}\}.
\]

Strong uniqueness of the parametrization process implies that the paths of $(X^{x}+kl\hat{e})_{t\ge 0}$ and $(X^{x+kl\hat{e}})_{t\ge 0}$ coincide almost surely. In particular, for $V\in \mathcal{B}_{E}$
\[
    \mathbb{P}\left[X^{x+kl\hat{e}}(t)\in f^{-1}(V)\right] 
    =\mathbb{P}\left[X^{x}(t)\in \{f^{-1}(V)-kl\hat{e}\}\right].
\]
Due to decomposition property above the preimage $f^{-1}(V)$ is invariant under translations along the vector $l\hat{e}$, that is, $f^{-1}(V)-kl\hat{e} = f^{-1}(V)$ as sets. We obtain the following relation for the transition probability function of the process $X$: For all $V\in \mathcal{B}_{E}$, and $x,\tilde x\in \overline D$ such that $f(x) = f(\tilde x)$
\begin{equation}\label{eq::transition_function_Dynkin}
P(x,t,f^{-1}(V)) 
    = P(\tilde x,t,f^{-1}(V)).    
\end{equation}
\begin{definition}[Quotient parametrization process] A quotient parametrization process started at $x\in E$ is a process in $E$ given by 
\[
    \tilde{X}^{x}(t) = f(X^{x}(t)),
\]    
where $f$ is as in \eqref{eq::function_f_D_to_E} and $X^{x}$ is a parametrization process started at $x\in E\subset D$.
\end{definition}

By \cite[Theorem 10.13]{dynkin1965markovI}, one can construct a filtered probability space $(\Omega, (\tilde{\mathcal{F}}_{t})_{t\ge 0}, \tilde{\mathbb{P}})$ such that $(\tilde{X}^{x},\tilde{\mathbb{P}})$ is a strong Markov-Feller process in the state space $(E, \mathcal{B}_{E})$, started at $x\in E$, with the transition probability function given by 
\[
    \tilde P(x, t, V) = P(x,t,f^{-1}(V)) = \sum\limits_{k=-\infty}^{\infty} P(x,t,V + kl\hat{e}).
\]
\subsubsection*{The process on the curve}
Let $\boldsymbol{z} = (z_{1},\dots,z_{N})$ denote a point in $\mathbb{C}^{N}$. Dyson Brownian motion on the curve is the process $\boldsymbol{Z} = (Z_{1}(t), \dots, Z_{N}(t))_{t\ge 0}$, given by 
\[
    Z_{i}(t) = \gamma(\tilde{X}_{i}^{x}(t)),
\]
in the state space
\[
S = \{(z_{1},\dots,z_{N})\in \Gamma^{N}: z_{i} = \gamma(x_{i}),\quad  x\in E\}.
\]
Denote by $g:E\to \Gamma^{N}$ the mapping 
\[
    g(x) = (\gamma(x_{1}), \dots, \gamma(x_{N})).    
\]
If $S = g(E)$, then $g:E\to S$ is a continuous bijective mapping. By \cite[Theorem 10.13]{dynkin1965markovI}, one can construct a filtered probability space $(\Omega, (\mathcal{G}_{t})_{t\ge 0}, \mathbb{Q})$ such that $(\boldsymbol{Z},\mathbb{Q})$ is a strong Markov-Feller process in the state space $(S, \mathcal{B}_{S})$, started at $\boldsymbol{z} = f(x)\in S$, with the transition probability function given by 
\[
    Q(\boldsymbol{z}, t, \mathcal{\mathcal{V}}) = \tilde{P}(x, t, g^{-1}(\mathcal{V}))=P(x,t,f^{-1}(g^{-1}(\mathcal{V}))), \quad \mathcal{V}\in \mathcal{B}_{S}.
\]
This completes the proof of Theorem~\ref{thm:existence-intro}.

\section{Transition probability functions}
In this section we consider the parametrization process for Dyson Brownian motion and study its transition probability function.
\subsection{Fokker-Planck-Kolmogorov equation}\label{sec::FPK_and_ergodicity}
Let $\Omega\subset\mathbb{R}^{N}$ be a domain, $b=(b_{1},\dots,b_{N})$ be a Borel vector field on $\Omega$, $A = (a_{ij})$ be matrix-valued mapping on $\Omega$ such that $a_{ij}$ are Borel measurable, $a_{ij}=a_{ji}$, and $A\ge0$. Define $L_{A,b}$ as the differential operator that acts on $\varphi\in C_{0}^{\infty}((0,T)\times\Omega)$ by
\begin{align*}
    L_{A,b}\varphi = \sum\limits_{i,j=1}^{N}a_{ij}\frac{\partial^2\varphi}{\partial x_{i}\partial x_{j}} + \sum\limits_{i=1}^{N}b_{i}\frac{\partial\varphi}{\partial x_{i}},
\end{align*}
and denote by $L^*_{A,b}$ its formal adjoint.
\begin{definition}
(a) A locally finite Borel measure $\mu$ on the domain $(0,T)\times\Omega\subset \mathbb{R}^{N+1}$ satisfies Fokker-Planck-Kolmogorov equation $\partial_{t}\mu = L^{*}_{A,b}\mu$
if $a_{ij}, b_{i}\in L^{1}_{loc}(|\mu|)$ and
\begin{align*}
    \int_{(0,T)\times\Omega}(\partial_{t}\varphi+L_{A,b}\varphi)d\mu
    =0 
    \quad \forall\varphi\in C_{0}^{\infty}((0,T)\times\Omega).
\end{align*}
(b) A locally finite Borel measure $\mu$ on a domain $\Omega\subset \mathbb{R}^{N}$ satisfies the stationary Fokker-Planck-Kolmogorov equation  $L^{*}_{A,b}\mu = 0$
if $a_{ij}, b_{i}\in L^{1}_{loc}(|\mu|)$ and
\begin{align*}
    \int_{\Omega}(L_{A,b}\varphi)(x) \mu(dx)=0 \quad \forall\varphi\in C_{0}^{\infty}(\Omega).
\end{align*}
\end{definition} 

In the rest, we choose $A = \frac{1}{2}I$ and denote by $L = L_{\frac{1}{2}I,b}$, where $b = \frac{1}{2}\nabla \log\rho$.  The differential operator $L$ is given by 
\begin{align*}
    Lf = \frac{1}{2}\Delta f + b\cdot \nabla f, \qquad f\in C_{0}^{\infty}(\Omega),
\end{align*}
and $L^{*}$ denotes its formal adjoint
\begin{align*}
    L^{*}f = \frac{1}{2}\Delta f- \nabla\cdot (b f).
\end{align*}

In the special case $d\mu = \mu_{t}(dy)dt$, one can consider a Cauchy problem.
\begin{definition}[Cauchy problem]
Let $\nu$ be a locally bounded measure on a domain $\Omega$. A measure $d\mu = \mu_{t}(dy)dt$ solves a Cauchy problem with the initial condition $\mu_{0}= \nu$ if 
\begin{equation}\label{eq::weak_FPK_measure}
    \int_{0}^{T}\int_{\Omega}(\partial_{t}\varphi+ L\varphi)\mu_{t}(dy)dt 
    = 0 \quad \forall\varphi\in C^{\infty}_{0}((0,T)\times \Omega)
\end{equation}
and
\begin{equation}\label{eq::initial_cond_measure}
    \int_{\Omega}\varphi(x)\nu(dx)
    =\lim\limits_{t\to 0+, t\in I_{\varphi}}\int_{\Omega}\varphi(x)\mu_{t}(dx) 
    \quad\forall\varphi\in C^{\infty}_{0}(\Omega),
\end{equation}
where the set $I_{\varphi}\subset (0,T)$ has full measure $|I_{\varphi}|=T$
\end{definition}
\begin{remark}
In general, $I_{\varphi}$ depends on the test function. However, when $t\to\int_{\Omega}\varphi d\mu_{t}$ is continuous for all test functions, which is the case when $\mu_{t}(dy)$ is a transition probability of a Feller process, one can take $I_{\varphi} = (0,T)$.

\end{remark}
\begin{remark}
An equivalent definition of the Cauchy problem (\ref{eq::weak_FPK_measure}), (\ref{eq::initial_cond_measure}) is
\[
    \int_{\Omega}\varphi d\mu_{t} 
    = \int_{\Omega}\varphi d\nu + \int_{0}^{t}\int_{\Omega}L\varphi d\mu_{s}ds
    \quad\forall \varphi\in C^{\infty}_{0}(\Omega),
\]
which for the initial distribution $\nu = \delta_{x}$, i.e., the point mass at $x\in \Omega$, becomes
\begin{equation}\label{eq::Cauchy_probem_x}
    \int_{\Omega}\varphi d\mu_{t} 
    = \varphi(x) + \int_{0}^{t}\int_{\Omega}L\varphi d\mu_{s}ds
    \quad\forall \varphi\in C^{\infty}_{0}(\Omega).
\end{equation}    
\end{remark}
\begin{lemma}
If $\rho\in W^{1,1}_{loc}(\Omega)$, then the measure $\mu(dx) = \rho(x) dx$ satisfies the stationary Fokker-Planck-Kolmogorov equation $L^{*}\mu=0$ with the drift $ b = \frac{1}{2}\nabla\log\rho$, where the gradient is understood in the weak sense and $b(x)=0$ whenever $\rho(x)=0$.
\end{lemma}
\begin{proof}
    Indeed, since  $\rho\in W_{loc}^{1,1}(\Omega)$, for any compact subset $K\subset\Omega$  we have $$\int_{K} |b(x)| |\mu|(dx) = \frac{1}{2}\int_{K} |\nabla\rho|dx<\infty,$$
that is, $b^{i}\in L^{1}_{loc}(|\mu|)$ for all $i=1,\dots,N$. By integration by parts, for all $\varphi\in C_{0}^{\infty}(\Omega)$,  
\[
    \int_{\Omega}\left(\frac{1}{2} \Delta \varphi+b \cdot \nabla\varphi\right)\rho \ dx 
    =-\frac{1}{2}\int_{\Omega}\nabla\varphi\cdot \nabla\rho\ dx 
    + \int_{\Omega} \frac{\nabla\rho}{2\rho}\cdot \nabla\varphi \rho \ dx = 0. \qedhere
\]
\end{proof} 
\begin{corollary}
For any rectifiable Jordan curve $\Gamma$, the measure $\mu(dx)=\rho_{\beta,N}(x)dx$ on any domain $\Omega\subseteq D$, with the density 
\[
    \rho_{\beta,N}(x) = \frac{1}{Z_{\beta, N}(\gamma)}\prod\limits_{i\neq j}|\gamma(x_{i})-\gamma(x_{j})|^{\beta/2},
\]
is a solution to the stationary Fokker-Planck-Kolmogorov equation $L^{*}\mu=0$ with the drift 
\[
b_{i}(x) 
= \frac{1}{2\rho_{\beta,N}}\frac{\partial\rho_{\beta,N}}{\partial x_{i}}(x) 
= \frac{\beta}{2}\sum\limits_{j:j\neq i}\Re\left\{\frac{\gamma'(x_{i})}{\gamma(x_{i})-\gamma(x_{j})}\right\}, 
\]
where the derivative of the parametrization is understood in the weak sense since it exists almost everywhere.
\end{corollary}
\begin{remark}
Note that in general for gradient drifts there could exist infinitely many solutions to the stationary Fokker-Planck-Kolmogorov equation even in the class of probability measures. See an example in \cite[p.790]{bogachev2010invariant}. In Proposition~\ref{prop:unique_station_FPK_smooth_case} of the appendix we show that, at least when the density $\rho$ is smooth in $D$, any positive solution to the stationary Fokker-Planck-Kolmogorov equation must equal $\rho$ times a constant. 
\end{remark}
\subsection{Parametrization process}
\subsubsection*{Fokker-Planck-Kolmogorov equation}
In the following proposition we show that the transition probability function of the parametrization process has a density and satisfies the Fokker-Planck-Kolmogorov equation.
\begin{proposition}\label{prop::FPK_X_process}
The transition probability function $P(x,t,dy)$ of the parametrization process $X^{x}$ solves the Cauchy problem \eqref{eq::Cauchy_probem_x} with $\rho=\rho_{\beta, N}$. Moreover, the transition probability measure is absolutely continuous with respect to the Lebesgue measure 
\begin{align*}
    P(x,t,dy) = p(x,t,y)dy,
\end{align*}
$(t,y)\to p(x,t,y)$ is positive and locally Hölder continuous on $(0,+\infty)\times D$, and the density satisfies the Fokker-Planck-Kolmogorov equation with reflecting boundary conditions
\begin{align*}
    &\partial_{t}p(x,t,\bullet)=L^{*}p(x,t,\bullet) \text{ weakly in }D,\\
    &n\cdot\left(-\frac{1}{2}\rho_{\beta, N}\nabla \frac{p}{\rho_{\beta, N}} \right) = 0 \text{ weakly on }\partial D.
\end{align*}
% In addition, the ratio
% \[
%    (t,x)\mapsto \sup_{y\in D}\frac{p(x,t,y)}{\rho_{\beta, N}(y)}
% \]
% is locally bounded on $(0,+\infty)\times D$.
\end{proposition}
In the proof of Proposition~\ref{prop::FPK_X_process} we will use the following result about existence of densities for measure-valued solutions to the Fokker-Planck-Kolmogorov equation.
\begin{lemma}[\cite{bogachev2015FPK}]\label{lem::existence_of_density}
Let $d\mu = \mu_{t}(dy)dt$ be a locally finite non-negative Borel measure on $(0,T)\times \Omega$ which solves the Fokker-Planck-Kolmogorov equation $\partial_{t}\mu = L^{*}\mu$, with $b_{i}\in L^\infty_{loc}(\Omega )$. Then, 
$d\mu = p(t,y)dydt$, where $p\in L^{r}_{loc}((0,T)\times \Omega)$ for any $r\in [1,1+\frac{1}{N+1}]$. The density $p$ is a non-negative locally Hölder continuous function on $(0,T)\times \Omega$.
\end{lemma}

\begin{proof}[Proof of Proposition~\ref{prop::FPK_X_process}]
$1^{\circ}.$ The drift $b = \frac{1}{2}\nabla\log\rho_{\beta,N}$ is locally bounded in $D$, so in particular $b_{i}\in L^{1}_{loc}(P(x,t,dy))$. By Itô's formula, for any $\varphi\in C^{2}({D})$, almost surely
\begin{align*}
    \varphi(X^{x}(t))
    = \varphi(x) + \int_{0}^{t}(L\varphi)(X^{x}(s))ds + \int_{0}^{t}\nabla\varphi (X^{x}(s))\cdot dB(s).
\end{align*}
Taking the expectation shows that $P(x,t,dy)$ solves the Cauchy problem 
\[
    \int_{D}\varphi(y)P(x,t,dy)  = \varphi(x) + \int_{0}^{t}\int_{D}(L\varphi)(y)P(x,t,dy).
\]
In particular,  differentiating with respect to time variable gives the weak formulation of the Fokker-Planck-Kolmogorov equation
\begin{align*}
    \frac{d}{dt}\int_{D}\varphi(y)P(x,t,dy) = \int_{D}(L\varphi)(y)P(x,t,dy).
\end{align*}

$2^{\circ}.$ By Lemma~\ref{lem::existence_of_density} the measure $P(x,t,dy)dt$ has a density with respect to $dydt$, which is positive and locally Hölder continuous on $(0,T)\times D$. This implies that $P(x,t,dy)\ll dy$ for almost every $t\in (0,T)$. Since the process is Feller,  the function $t\to P(x,t,B)$ is right-continuous for any $B\in\mathcal{B}_{D}$. Hence, $P(x,t,dy) \ll dy$ for all $t\in (0,T)$, and, as it holds for any $T>0$, the transition probability $P(x,t,dy)$ has a density for all $t>0$:
\[
    P(x,t,dy) = p(x,t,y)dt.
\]

$3^{\circ}.$ $L^{*}p(x,t,\bullet)$ is defined in the weak sense by the identity
\begin{align*}
    \int_{D}\varphi (L^{*}p)dy = \int_{D}(L\varphi)pdy\quad\forall\varphi\in C^{2}_{0}(D).
\end{align*}
Thus, it follows from $1^{\circ}.$ that $\partial_{t}p = L^{*}p$ weakly in $D$.
Let $K$ be a compact subset of $D$ with Lipschitz boundary. By Green's identity
\begin{align*}
    \int_{K}(Lf)gdy + \frac{1}{2}\int_{\partial K } (n\cdot\nabla f)gdS 
    = \int_{K}f (L^{*}g)dy 
    + \int_{\partial K}f n\cdot \left(-\frac{1}{2}\nabla g 
    + \frac{\nabla\rho}{2\rho} g\right)d\sigma\quad \forall g,f \in C^{2}(K).
\end{align*}
In general, $J = -\frac{1}{2}\rho\nabla \frac{p}{\rho}$ might not make sense pointwise, but $n\cdot J|_{\partial K}$ can be defined weakly by
\begin{align*}
    \int_{\partial K}f (n\cdot J)d\sigma 
    = \int_{K}(Lf)pdy 
    - \int_{K}f (L^{*}p)dy 
    + \frac{1}{2}\int_{\partial K } (n\cdot\nabla f)pd\sigma\quad \forall f \in C^{2}(K).
\end{align*}
Let $f\in C^{2}(\overline{D})$ and $n\cdot\nabla f=0$ on $\partial D$. Then, taking into account that $\partial_{t}p = L^{*}p$ weakly in $D$, in the limit $K\uparrow D$
\begin{align*}
    \lim\limits_{K\uparrow D}\int_{\partial K}f (n\cdot J)d\sigma  
    =\int_{D}(Lf)pdy
    - \int_{D}f (\partial_{t}p)dy
    + \frac{1}{2}\int_{\partial D} (n\cdot\nabla f)pd\sigma.
\end{align*}
The first two terms cancel each other due to $1^{\circ}.$, and the last term is zero because of the choice of the test function. Hence, the reflecting condition $n\cdot J|_{\partial D}=0$ is understood as 
\begin{align*}
    \lim\limits_{K\uparrow D}\int_{\partial K}f (n\cdot J)d\sigma  
    =0 \quad \forall f\in C^{2}(\overline{D}).
\end{align*}
\end{proof}
\subsubsection*{Kolmogorov type diffusion}
In the section we demonstrate that the parametrization process satisfies Kolmogorov's definition of diffusion, i.e., a Markov process with transition probability function $P(x,t,dy)$ such that there exit a function $b:D\times [0,+\infty)\to\mathbb{R}^{N}$, called drift coefficient, and a matrix-valued function $a(x,t)$, called diffusion coefficient, and for all $\delta>0, x,z\in D$ such that $\overline{B(x,\delta)}\subset D$ it holds that 
\begin{enumerate}
    \item $\lim\limits_{h\to 0+}\frac{1}{h}\int_{B(x,\delta)^{c}}p(x,h,y)dy=0,$
    \item $\lim\limits_{h\to 0+}\frac{1}{h}\int_{B(x,\delta)}(y-x)p(x,h,y)dy = b(x,t),$
    \item $\lim\limits_{h\to 0+}\frac{1}{h}\int_{B(x,\delta)}((y-x)\cdot z )^2p(x,h,y)dy= a(x,t)z\cdot z.$
\end{enumerate}  
\begin{proposition}\label{prop::Kolmogorov_diffusion}
Suppose $\gamma \in C^{2}$.
Then, the parametrization process $X^{x}$ is a diffusion process in the sense of Kolmogorov with identity diffusion matrix and the drift coefficient given by $b =\frac{1}{2}\nabla\log\rho_{\beta, N}$.   
\end{proposition}
\begin{proof}
$1^{\circ}.$ Let $f\in C^{\infty}(D)$. The time derivative of $(P_{t}f)(x)=\mathbb{E}[X^{x}(t)]$ at $t=0$ can be written as 
    \begin{align*}
        \frac{d}{dt}\int_{ D}f(y)P(x,t,dy)\Big|_{t=0}
        &= \lim\limits_{h\to 0+}\frac{1}{h}\left(\int_{ D}f(y)P(x,h,dy) - \int_{ D}f(y')P(x,0,dy')\right) \\ 
        &=\lim\limits_{h\to 0+}\frac{1}{h}\left(\int_{ D}f(y)P(x,h,dy) -f(x)\right)  \\
        & = \lim\limits_{h\to 0+}\frac{1}{h}\int_{ D}\left(f(y)-f(x)\right)P(x,h,dy) .
    \end{align*}
On the other hand, as in the proof of Proposition~\ref{prop::FPK_X_process}, from Itô's formula we have
\begin{align*}
     \frac{d}{dt}\int_{ D}f(y)P(x,t,dy)\Big|_{t=0}
     = \int_{D}(Lf)(y)P(x,t,dy)\Big|_{t=0}
     = (Lf)(x).
\end{align*}
We obtain the identity
\[
    \lim\limits_{h\to 0+}\frac{1}{h}\int_{ D}\left(f(y)-f(x)\right)P(x,h,dy)  = (Lf)(x) \quad\forall f\in C^{\infty}(D).
\]
Taking $f_{k}(y)=y_{k}$ yields
\begin{align*}
\lim\limits_{h\to 0+}\frac{1}{h}\int_{ D}\left(y_{k}-x_{k}\right)P(x,h,dy) 
=(Lf_{k})(x)=b_{k}(x).
\end{align*}
Taking $f_{ij}(y) =(y_{i}-x_{i})(y_{j}-x_{j})z_{i}z_{j}$ yields
\begin{align*}
   \lim\limits_{h\to 0+}\frac{1}{h}\int_{ D}(y_{i}-x_{i})(y_{j}-x_{j})z_{i}z_{j}P(x,h,dy) =  (Lf_{ij})(x) = z_{i}^2\delta_{ij}.
\end{align*}

$2^{\circ}.$ Next we show stochastic continuity of the transition probability function, that is, for all $x\in D$ and $\delta > 0$ such that $\overline{B(x,\delta)}\subset D$ the following limit holds
\begin{align*}
    \lim\limits_{h\to 0+}\frac{1}{h}P(x,h,B(x,\delta)^{c})=0.
\end{align*}
By Itô's formula, for $p>2$,
\begin{align*}
    \mathbb{E}\left[|X^{x}(t)-x|^{p}\right] 
    = &p\frac{N+p-2}{2}\int_{0}^{t}\mathbb{E}\left[|X^{x}(s)-x|^{p-2}\right]ds \\
    &+ p \int_{0}^{t}\mathbb{E}\left[|X^{x}(s)-x|^{p-2}\sum\limits_{i=1}^{N}b_{i}(X^{x}(s))((X^{x}(s))_{i}-x_{i})\right]ds.
\end{align*}
The sum in the second term can be written as
\begin{align*}
    \sum\limits_{i=1}^{N}b_{i}(y)(y_{i}-x_{i}) 
    = \sum\limits_{i=1}^{N}(b_{i}(y)-b_{i}(x))(y_{i}-x_{i})  +\sum\limits_{i=1}^{N} (y_{i}-x_{i})b_{i}(x). 
\end{align*}
Now we will show that the first term on the right hand side is bounded from above.
\begin{align*}
     &\sum\limits_{i=1}^{N}(b_{i}(y)-b_{i}(x))(y_{i}-x_{i}) \\
     &= \frac{\beta}{2}\sum\limits_{i=1}^{N}(y_{i}-x_{i})\sum\limits_{j:j\neq i}\Re\left\{\frac{\gamma'(y_{i})}{\gamma(y_{i})-\gamma(y_{j})}-\frac{\gamma'(x_{i})}{\gamma(x_{i})-\gamma(x_{j})} \right\}\\
     &= \frac{\beta}{2}\sum_{i < j}\Re\left\{\frac{(y_{i}-x_{i})\gamma'(y_{i})-(y_{j}-x_{j})\gamma'(y_{j})}{\gamma(y_{i})-\gamma(y_{j})}\right\}
    + \frac{\beta}{2} \sum_{i< j}\Re\left\{\frac{(x_{i}-y_{i})\gamma'(x_{i})-(x_{j}-y_{j})\gamma'(x_{j})}{\gamma(x_{i})-\gamma(x_{j})}\right\}.
\end{align*}
Now we take a closer look at the first term, which can be factored as
\begin{align}
    &\Re\left\{\frac{(y_{i}-x_{i})\gamma'(y_{i})-(y_{j}-x_{j})\gamma'(y_{j})}{\gamma(y_{i})-\gamma(y_{j})}\right\} \\
    &= (y_{i}-x_{i})\Re\left\{\frac{\gamma'(y_{i})-\gamma'(y_{j})}{\gamma(y_{i})-\gamma(y_{j})}\right\}
    + \left(1 - \frac{x_{j}-x_{i}}{y_{j}-y_{i}}\right)\Re\left\{\gamma'(y_{j})\frac{y_{i}-y_{j}}{\gamma(y_{i})-\gamma(y_{j})}\right\}.\label{eq::monotone_drift_diverges}
\end{align}
Since $x,y\in D$, it holds that $x_{i}<x_{j}$ and $y_{i}<y_{j}$ for $i<j$, and the term $\frac{x_{i}-x_{j}}{y_{i}-y_{j}}$ is always positive.  If $y_{j}-y_{i}\to0+$, the expression \eqref{eq::monotone_drift_diverges} can only go to $-\infty$, because under the assumption $\gamma\in C^2$
\[
    \lim\limits_{y_{j}-y_{i}\to 0+}\Re\left\{\gamma'(y_{j})\frac{y_{i}-y_{j}}{\gamma(y_{i})-\gamma(y_{j})}\right\} = 1
\]
and 
\[
     \lim\limits_{y_{j}-y_{i}\to 0+}  \Re\left\{\frac{\gamma'(y_{i})-\gamma'(y_{j})}{\gamma(y_{i})-\gamma(y_{j})}\right\} 
     = \Re\left\{\frac{\gamma''(y_{j})}{\gamma'(y_{j})}\right\} 
     =  \Re\left\{\gamma''(y_{j}) \overline{\gamma'(y_{j})}\right\} 
     =0.
\]
The same conclusion holds when $y_{j}-y_{i}\to l$ since the expression \eqref{eq::monotone_drift_diverges} can be rewritten as
\begin{align*}
    &\Re\left\{\frac{(y_{i}-x_{i})\gamma'(y_{i})-(y_{j}-x_{j})\gamma'(y_{j})}{\gamma(y_{i})-\gamma(y_{j})}\right\} \\
    &= (y_{i}-x_{i})\Re\left\{\frac{\gamma'(y_{i})-\gamma'(y_{j})}{\gamma(y_{i})-\gamma(y_{j})}\right\}+ \left(1 - \frac{l+x_{i}-x_{j}}{l+y_{i}-y_{j}}\right)\Re\left\{\gamma'(y_{j})\frac{l+y_{i}-y_{j}}{\gamma(y_{i})-\gamma(y_{j})}\right\}.
\end{align*}
Hence, 
\begin{align*}
    \lim\limits_{y\to \partial D}(b(x)-b(y))\cdot (x-y)= -\infty,
\end{align*}
and similarly for $x\to\partial D$. There exist a constant $C=C(\gamma, N)>0$ such that 
\begin{align*}
    \sum\limits_{i=1}^{N}(b_{i}(y)-b_{i}(x))(y_{i}-x_{i}) < C(\gamma, N).
\end{align*}
Hence, there are positive constants $c_{1},c_{2}$, that depend on $N,p, C(\gamma, N), |b(x)|, \beta $, such that 
\begin{align*}
    \mathbb{E}\left[|X^{x}(h)-x|^{p}\right] 
    \le c_{1} \int_{0}^{h}\mathbb{E}\left[|X^{x}(s)-x|^{p-2}\right]ds 
    + c_{2} \int_{0}^{h}\mathbb{E}\left[|X^{x}(s)-x|^{p-1}\right]ds.
\end{align*}
This implies that 
\[
    \lim\limits_{h\to 0+}\frac{1}{h}\mathbb{E}\left[|X^{x}(h)-x|^{p}\right] =0.
\]
By Chebyshev's inequality
\[
    P(x,h,B(x,\delta)^{c}) 
    = \mathbb{P}[|X^{x}(h)-x|\ge \delta]
    \le \delta^{-p}\mathbb{E}\left[|X^{x}(h)-x|^{p}\right].
\]
Hence, for all $x\in D$ and $\delta > 0$ such that $\overline{B(x,\delta)}\subset D$ we have
\begin{align*}
    \lim\limits_{h\to 0+}\frac{1}{h}P(x,h,B(x,\delta)^{c})=0.
\end{align*}
Consequently, combining the stochastic continuity with the results of $1^{\circ}.$ gives items $1.$ and $2.$ of the proposition. The fact that the equations $1.-3.$ can be written in terms of densities of the transition probabilities $P(x,t,dy)$ follows from Proposition~\ref{prop::FPK_X_process}.
\end{proof}

\subsection{Quotient parametrization process}
Let $\tilde X^{x}$ be the quotient parametrization process, defined in Section~\ref{sec::back_to_the_curve}, with the state space  
\begin{equation}\label{eq:domain_E}
    E=\{x\in\mathbb{R}^{N}: 0\le x_{1} < \dots< x_{N}< x_{1}+l<2l\},
\end{equation}
and a starting point $x\in E$. The transition probability function of $\tilde X^{x}$ is given by 
\[
    \tilde P(x,t, B) = P(x,t,f^{-1}(B)), \quad B\in\mathcal{B}_{E},
\]
where $P(x,t,dy)$ is the transition probability function of the parametrization process.
\begin{proposition}\label{prop::FPK_quotient_param_process}
    The transition probability function $\tilde P(x,t,dy)$ solves the Cauchy problem for the Fokker-Planck-Kolmogorov equation in $E$ with the drift $b_{\beta,N}=\frac{1}{2}\nabla\log\rho_{\beta, N}$. Moreover, $\tilde P(x,t,dy)$ has a positive locally Hölder continuous density 
    \[
        \tilde P(x,t,dy) = \rho_{t}(x,y)dy
    \]
    for every $t>0$.
\end{proposition}
\begin{proof}
For any $\varphi \in C^{\infty}_{0}(E)$, the mapping 
\[
    y\to \varphi(f(y)) = \sum\limits_{k=-\infty}^{\infty} \varphi( y-kl\hat{e})\mathds{1}_{\{E+kl\hat{e}\}}(y)
\]
is $C^{\infty}(D)$, and the following associativity relation holds 
\[
    (L\varphi)\circ f = L(\varphi\circ f),
\]
where $L=\frac{1}{2}\Delta + \frac{1}{2}\nabla\log\rho_{\beta, N}\cdot\nabla$. With the help of Proposition~\ref{prop::FPK_X_process} and the change of variable formula we obtain
\begin{align*}
    \int_{E}\varphi(\tilde y)\tilde P(x,t, d\tilde y)  
    &=\int_{D}(\varphi\circ f)(y)P(x,t,dy)\\
    &= (\varphi\circ f)(x) + \int_{0}^{t}\int_{D}(L(\varphi\circ f))(y)P(x,s,dy)ds\\
    &=\varphi(x) + \int_{0}^{t}\int_{D}((L\varphi)\circ f)(y)P(x,s,dy)ds\\
    &=\varphi(x) + \int_{0}^{t}\int_{E}(L\varphi)(\tilde y)\tilde P(x,s,d\tilde y)ds.
\end{align*}
Therefore, the transition probability function $\tilde P(x,t,dy)$ solves the Cauchy problem for the  FPK equation on $E$ with the initial distribution $\delta_{x}$. Since the process $\tilde X^{x}$ is Feller, as in the proof of Proposition~\ref{prop::FPK_X_process}, the transition probability $\tilde P(x,t,dy)$ has a positive locally Hölder continuous density 
\[
    \tilde P(x,t,dy) = \rho_{t}(x,y)dy
\]
for every $t>0$.
\end{proof}

\begin{remark}
    If the parametrization $\gamma \in C^{\infty}$, then the drift coefficient $b_{\beta,N}=\frac{1}{2}\nabla\log\rho_{\beta, N}$ is smooth in $E$. Hence, by Weyl regularity, the density $(t,y)\to\rho_{t}(x,y)$, as a solution to the FPK equation with smooth coefficients, is itself smooth in $E$ and satisfies the FPK equation pointwise
\[
    \partial_{t}\rho_{t} = \nabla_{y}\cdot \left(-\frac{1}{2}\rho_{\beta, N}\nabla_{y}\frac{\rho_{t}}{\rho_{\beta, N}}\right).
\]
\end{remark}

\section{Convergence to stationary distribution}\label{sec::ergodicity}
In this section we prove Theorem~\ref{thm::exp_conv_to_stationary_distribution}, that is, the transition probability of Dyson Brownian motion converges exponentially fast, as $t\to+\infty$, to the stationary distribution given by the Coulomb gas density $\rho_{\beta, N}$. First, we show that the transition probability of the quotient parametrization process $\tilde{X}^{x}$, defined in Section~\ref{sec::back_to_the_curve}, satisfies a Poincaré type inequality. Then, using the inequality and Fokker-Plank-Kolmogorov equation, we prove convergence and estimate the convergence rate. 

\subsection{Poincaré type inequality}
\begin{lemma}\label{lem::Poincare_inequality} Let $\gamma$ be $C^{1}$ and let $E$ be as in \eqref{eq:domain_E}. Then, there exists a constant $C = C(\beta, \gamma, N)>0$ such that for any $f\in C^{1}(E)$ such that $\int_{E}f(y)\rho_{\beta,N}(y)dy=0$ the following Poincaré type inequality holds
\[
    \int_{E}f^{2}(y)\rho_{\beta,N}(y)dy \le C \int_{E} \left|\nabla f(t)\right|^{2}\rho_{\beta,N}(y)dy.
\]
\end{lemma}
\begin{proof}
By Jensen inequality, 
\begin{align*}
    \int_{E}f(y)^{2}\rho_{\beta,N}(y)dy 
    &= \int_{E}\left(f(y)- \int_{E}f(\tilde y)\rho_{\beta,N}(\tilde y)d\tilde y\right)^{2}\rho_{\beta,N}(y)dy\\
    &\le \iint_{E\times E}(f(y)-f(\tilde y))^{2}\rho_{\beta,N}(y)\rho_{\beta,N}(\tilde y)dyd\tilde y\\
    & = 2\iint_{E\times E}(f(y)-f(\tilde y))^{2}\rho_{\beta,N}(y)\rho_{\beta,N}(\tilde y)\mathds{1}_{\{y_{1} < \tilde  y_{1}\}}dyd\tilde y,
\end{align*}
\begin{align*}
    f(y)-f(\tilde y) 
    =& f(y_{1}, \dots, y_{N}) - f(\tilde y_{1}, y_{2},\dots,y_{N})\\
    &+ f(\tilde y_{1}, y_{2},\dots,y_{N}) - f(\tilde y_{1}, \tilde y_{2}, y_{3}, \dots,y_{N})\\
    &+\hspace{3.1cm}\dots\\
    &+ f(\tilde y_{1},\dots, \tilde y_{N-1},y_{N})-f(\tilde y_{1},\dots, \tilde y_{N-1},\tilde y_{N}).
\end{align*}
For $i=1,\dots, N$, denote
\begin{align*}
    a_{i}
    :&=f(\tilde y_{1},\dots \tilde y_{i-1},  y_{i}, y_{i+1}, \dots, y_{N}) - f(\tilde y_{1},\dots,  \tilde y_{i-1}, \tilde y_{i}, y_{i+1}, \dots, y_{N}) \\
    &= \int_{\tilde y_{i}}^{y_{i}} \partial_{r_{i}}f(\tilde y_{1},\dots \tilde y_{i-1},  r_{i}, y_{i+1}, \dots, y_{N}) dr_{i}.
\end{align*}
Then, 
\begin{align*}
    \int_{E}f(y)^{2}\rho_{\beta,N}(y)dy  
    \le  2 N \sum_{i=1}^{N}\iint_{E\times E}(a_{i})^{2}\rho_{\beta,N}(y)\rho_{\beta,N}(\tilde y)\mathds{1}_{\{y_{1} < \tilde  y_{1}\}}dyd\tilde y
\end{align*}
By Cauchy-Schwarz inequality,
\begin{align*}
    (a_{i})^{2} &=\left(\int_{\tilde y_{i}}^{y_{i}} \partial_{r_{i}}f(\tilde y_{1},\dots \tilde y_{i-1},  r_{i}, y_{i+1}, \dots, y_{N}) dr_{i}\right)^{2} \\
    &\le \int_{\tilde y_{i}}^{y_{i}} |\nabla f(\tilde y_{1},\dots \tilde y_{i-1},  r_{i}, y_{i+1}, \dots, y_{N})|^{2} \rho_{\beta,N}(\tilde y_{1},\dots \tilde y_{i-1},  r_{i}, y_{i+1}, \dots, y_{N}) dr_{i}\times \\
    &\hspace{7cm}\times\int_{\tilde y_{i}}^{y_{i}} \frac{1}{\rho_{\beta,N}(\tilde y_{1},\dots \tilde y_{i-1},  r'_{i}, y_{i+1}, \dots, y_{N}) }dr'_{i}.
\end{align*}
\[
    J(y, \tilde y, r_{i}, r'_{i}):= \frac{ |\nabla f(\tilde y_{1},\dots \tilde y_{i-1},  r_{i}, y_{i+1}, \dots, y_{N})|^{2} \rho_{\beta,N}(\tilde y_{1},\dots \tilde y_{i-1},  r_{i}, y_{i+1}, \dots, y_{N}) }{\rho_{\beta,N}(\tilde y_{1},\dots \tilde y_{i-1},  r'_{i}, y_{i+1}, \dots, y_{N})}\rho_{\beta,N}(y)\rho_{\beta,N}(\tilde y).
\]
Taking into account that $y_{1} < \tilde  y_{1}$,
\begin{align*}
    &\iint_{E\times E} (a_{i})^{2}\rho_{\beta,N}(y)\rho_{\beta,N}(\tilde y)\mathds{1}_{\{y_{1} < \tilde  y_{1}\}}dyd\tilde y\\
    &\le
    \int_{0}^{l}dy_{1}
    \int_{y_{1}}^{l}d\tilde y_{1}\dots 
    %\int_{y_{1}}^{y_{1}+l}dy_{2}
    %\int_{\tilde y_{1}}^{\tilde y_{1}+l}d\tilde y_{2}\dots 
    \int_{y_{i-1}}^{y_{1}+l}dy_{i}
    \int_{\tilde y_{i-1}}^{\tilde y_{1}+l}d\tilde y_{i}
    \int_{\tilde y_{i}}^{y_{i}}dr_{i}
    \int_{\tilde y_{i}}^{y_{i}}d\tilde r_{i}\times\\
    &\hspace{2cm}\times\int_{y_{i}}^{y_{1}+l}dy_{i+1}
    \int_{\tilde y_{i}}^{\tilde y_{1}+l}d\tilde y_{i+1}\dots
    \int_{y_{N-1}}^{y_{1}+l}dy_{N}
    \int_{\tilde y_{N-1}}^{\tilde y_{1}+l}d\tilde y_{N}J(y,\tilde y, r_{i}, r'_{i})\\
    &\le
    \int_{0}^{l}dy_{1}
    \int_{y_{1}}^{l}d\tilde y_{1}\dots 
    %\int_{y_{1}}^{y_{1}+l}dy_{2}
    %\int_{\tilde y_{1}}^{\tilde y_{1}+l}d\tilde y_{2}\dots 
    \int_{y_{i-1}}^{y_{1}+l}dy_{i}
    \int_{\tilde y_{i-1}}^{\tilde y_{1}+l}d\tilde y_{i}
    \int_{\boldsymbol{\tilde y_{i-1}}}^{\boldsymbol{\tilde y_{1}+l}}dr_{i}
    \int_{\tilde y_{i}}^{y_{i}}d\tilde r_{i}\times\\
    &\hspace{2cm}\times
    \int_{\boldsymbol{r_{i}}}^{\boldsymbol{\tilde y_{1}+l}}dy_{i+1}
    \int_{\tilde y_{i}}^{\tilde y_{1}+l}d\tilde y_{i+1}\dots
    \int_{y_{N-1}}^{\boldsymbol{\tilde y_{1}+l}}dy_{N}
    \int_{\tilde y_{N-1}}^{\tilde y_{1}+l}d\tilde y_{N}
    J(y,\tilde y, r_{i}, r'_{i}).
\end{align*}
Denote $v = (\tilde y_{1}, \dots, \tilde y_{i-1}, r'_{i}, y_{i+1},\dots, y_{N})$. 
\begin{align*}
    &\int_{E}dv |\nabla f(v)|^{2}\rho_{\beta,N}(v) \times\\
    &\times
    \int_{0}^{l}dy_{1}\dots  
    \int_{y_{i-1}}^{y_{1}+l}dy_{i}
    \int_{\tilde y_{i-1}}^{\tilde y_{1}+l}d\tilde y_{i}
    \int_{\tilde y_{i}}^{y_{i}}d\tilde r_{i}
    \int_{\tilde y_{i}}^{\tilde y_{1}+l}d\tilde y_{i+1}\dots
    \int_{\tilde y_{N-1}}^{\tilde y_{1}+l}d\tilde y_{N} 
    \frac{\rho_{\beta,N}(\tilde y_{1},\dots, \tilde y_{N})\rho_{\beta,N}(y_{1},\dots,y_{N})}{\rho_{\beta,N}(\tilde y_{1},\dots, \tilde y_{i-1}, r'_{i}, y_{i+1},\dots, y_{N})}.
\end{align*}
The inner integrand equals
\begin{align*}
    \frac{\rho_{\beta,N}(\tilde y_{1},\dots, \tilde y_{N})\rho_{\beta,N}(y_{1},\dots,y_{N})}{\rho_{\beta,N}(\tilde y_{1},\dots, \tilde y_{i-1}, r'_{i}, y_{i+1},\dots, y_{N})}
    = & \frac{1}{Z_{\beta,N}(\gamma)}\frac{\prod_{k: k \neq i}|\gamma(\tilde y_{i})-\gamma(\tilde y_{k})|^{\beta}\prod_{l: l \neq i}|\gamma( y_{i})-\gamma( y_{l})|^{\beta}}{\prod_{p:p < i} |\gamma( r'_{i})-\gamma(\tilde y_{p})|^{\beta}
    \prod_{q:q > i} |\gamma( r'_{i})-\gamma(y_{q})|^{\beta}} \times \\
    &\times\prod_{i<\alpha < \eta}|\gamma(\tilde y_{\alpha})-\gamma(\tilde y_{\eta})|^{\beta}
    \prod_{\zeta < \theta < i}|\gamma( y_{\zeta})-\gamma( y_{\theta})|^{\beta}.
\end{align*}
For any $\beta >0$, the integral with respect to $y_{1}, \dots, y_{i}, \tilde y_{i}, r'_{i}, \tilde y_{i+1}, \dots, \tilde y_{N}$ is uniformly bounded with respect to the remaining variables by some constant $c = c(\beta, \gamma, N)>0$. Hence, 
\[
     \int_{E}f(y)^{2}\rho_{\beta,N}(y)dy 
     \le 2N^{2} c(\beta, \gamma, N)
     \int_{E} |\nabla f(v)|dv^{2}\rho_{\beta,N}(v).\qedhere
\]
\end{proof}
\begin{corollary}
 Suppose $\gamma\in C^{\infty}$.  Let $\rho_{t}(x,y)$ be the density of the transition probability function of the quotient parametrization process $\tilde X^{x}$ in $E$. Then, there exists a constant $C = C(\beta,\gamma, N)>0$ such that the following Poincaré type inequality holds
    \[
    \int_{E}\left(\frac{\rho_{t}(x,y)}{\rho_{\beta,N}(y)}-1\right)^{2}\rho_{\beta,N}(y)dy  \le C\int_{E}\left|\nabla_{y} \frac{\rho_{t}(x,y)}{\rho_{\beta,N}(y)}\right|^{2}\rho_{\beta,N}(y)dy.
\]
\begin{proof}
    The claim follows from Lemma~\ref{lem::Poincare_inequality} by taking $f(y) = \frac{\rho_{t}(x,y)}{\rho(y)}-1$.
\end{proof}
\end{corollary}
\subsection{Exponential convergence rate}
\begin{proposition}\label{prop::exp_convergence_of_density_in_E}
Suppose $\gamma\in C^{\infty}$. Let $\rho_{t}(x,y)$ be the density of the transition probability function of the quotient parametrization process $\tilde X^{x}$ in $E$.  Then, there exists a constant $C = C(\beta,\gamma, N)>0$ such that 
\begin{equation}\label{eq::Poincare_inequality}
    \int_{E}\left(\frac{\rho_{t}(x,y)}{\rho_{\beta,N}(y)}-1\right)^2\rho_{\beta,N}(y) dy
    \le \exp\left(-\frac{t}{C}\right).    
\end{equation}
In particular, the transition probability function $\tilde P(x, t, dy)$ converges weakly, as $t\to+\infty$, to the stationary distribution $\rho_{\beta,N}(y)dy$ on $(E, \mathcal{B}_{E})$ given by the density 
\[
    \rho_{\beta, N}(y) 
    = \frac{1}{Z_{\beta, N}(\gamma)}\prod_{i\neq j}|\gamma(y_{i})-\gamma(y_{j})|^{\frac{\beta}{2}}.
\]
\end{proposition}
\begin{proof}
In the course of the proof $\nabla \rho_{t}$ refers to taking the gradient with respect to the second spatial variable. Differentiating the left-hand side of the inequality \eqref{eq::Poincare_inequality} with respect to time gives
\begin{align*}
    \frac{d}{dt}\int_{E}\left(\frac{\rho_{t}}{\rho_{\beta,N}}-1\right)^2 \rho_{\beta,N}dy
    &=2 \int_{E}\left(\frac{\rho_{t}}{\rho_{\beta,N}}-1\right) \partial_{t}\rho_{t} dy\\
    &=-\int_{E}\left(\frac{\rho_{t}}{\rho_{\beta,N}}-1\right)\nabla\cdot\left(\rho_{\beta,N}\nabla\frac{\rho_{t}}{\rho_{\beta,N}}\right)dy,
\end{align*}
where in the last equality we used that the density $\rho_{t}$ satisfies the FPK equation pointwise in $E$. Integrating by parts, the right-hand side equals
\[
    -\frac{1}{2}\int_{E}\left|\nabla\frac{\rho_{t}}{\rho_{\beta,N}}\right|^{2}\rho_{\beta,N}dy  
    - \frac{1}{2}\int_{\partial E}\left(\frac{\rho_{t}}{\rho_{\beta,N}}-1\right) n\cdot \rho_{\beta,N}\nabla\frac{\rho_{t}}{\rho_{\beta,N}}d\sigma.
\]
The boundary term vanishes because $n\cdot \rho_{\beta,N}\nabla\frac{\rho_{t}}{\rho_{\beta,N}}$ is zero on that part of the boundary $\partial E$ that intersects with $\partial D$ due to reflecting boundary conditions, while on the remaining part of the boundary $\{x_{1}=0\}\cap\partial E$  and $\{x_{1}=l\}\cap \partial E$ the integrand takes opposite values due to $l$-periodicity of $\rho_{t}$ and $\rho_{\beta,N}$ along $\hat{e}$, and the fact that $n(x)=(1,0,\dots, 0)$ for $x\in \{x_{1}=0\}$ and $n(x)=(-1,0,\dots, 0)$ for $x\in \{x_{1}=l\}$.
\[
   \frac{d}{dt}\int_{E}\left(\frac{\rho_{t}}{\rho_{\beta,N}}-1\right)^2 \rho_{\beta,N}dy 
   =-\int_{E}\left|\nabla\frac{\rho_{t}}{\rho_{\beta,N}}\right|^{2}\rho_{\beta,N}dy. 
\]
The Poincaré inequality from Lemma~\ref{lem::Poincare_inequality} implies 
\[
    \frac{d}{dt}\int_{E}\left(\frac{\rho_{t}}{\rho_{\beta,N}}-1\right)^2 \rho_{\beta,N}dy
    \le - \frac{1}{C}\int_{E}\left(\frac{\rho_{t}}{\rho_{\beta,N}}-1\right)^2 \rho_{\beta,N}dy,
\]
which after integration yields the claim of the proposition.
\end{proof}
\noindent This result is immediately transferred to the transition probability function $Q(\boldsymbol{z}, t, |d\boldsymbol{w}|)$  of Dyson Brownian motion, which is given by the density 
\[
    \varrho_{t}(z,w) = \rho_{t}\left(\gamma^{-1}(z_{1}), \dots, \gamma^{-1}(z_{N}) , \gamma^{-1}(w_{1}), \dots, \gamma^{-1}(w_{N})\right).
\]
\begin{corollary}\label{cor::exp_convergence_process_on_Gamma}
Suppose $\gamma \in C^{\infty}$. There exists a constant $C=C(\beta,\gamma,N)>0$ such that 
\[
    \int_{S}\left(\frac{\varrho_{t}(\boldsymbol{z},\boldsymbol{w})}{\varrho_{\beta,N}(\boldsymbol{w})}-1\right)^2\varrho_{\beta,N}(\boldsymbol{w}) |d\boldsymbol{w}| 
    \le \exp\left(-\frac{t}{C}\right).
\]
In particular, the transition probability function $Q(\boldsymbol{z}, t, |d\boldsymbol{w}|)$ converges weakly, as $t\to+\infty$, to the stationary distribution $\varrho_{\beta,N}(\boldsymbol{w})|d\boldsymbol{w}|$ on $(S, \mathcal{B}_{S})$ given by the density 
\[
    \varrho_{\beta,N}(\boldsymbol{z}) 
    = \frac{1}{Z_{\beta, N}(\gamma)}\prod_{i\neq j}|z_{i}-z_{j}|^{\frac{\beta}{2}}.
\]
\end{corollary}
\begin{proof}
Proposition~\ref{prop::exp_convergence_of_density_in_E} and the change of variable formula, using $\gamma'=1$, immediately imply the corollary.\qedhere
\end{proof}
\noindent This concludes the proof of Theorem~\ref{thm::exp_conv_to_stationary_distribution}.

\section{Large deviations as \texorpdfstring{$\beta \to +\infty$}{β->+∞}}
We will now study the behavior of Dyson Brownian motion on $\Gamma$ as the inverse temperature parameter $\beta\to+\infty$. Since the process on the curve is a continuous image of the parametrization process, it is sufficient to study the system
\[
    dX(t) = dB(t)+\frac{1}{2}\nabla\rho_{\beta,N}(X(t))dt.
\]
After a deterministic time-change, this becomes
\[
    dX(t) = \sqrt{\kappa}dB(t) + \nabla\rho_{1,N}(X(t))dt
\]
with the small parameter $\kappa>0$ which is related to $\beta>0$ via
\[
    \kappa = \frac{2}{\beta}.
\]
We are interested in large deviations as $\kappa  \to 0$.
The recent article \cite{abuzaid2024large} establishes a large deviation principle for diffusions of ``Dyson-type'' with particles living on the unit circle, and the process is restricted to any compact time interval $[0,T]$. As we will see, this includes the parametrization process under suitable regularity assumptions on the curve $\Gamma$. First, we recall the definition of Dyson-type potential from \cite{abuzaid2024large} and show that, if one assumes the curve  $\gamma \in C^3$, the  parametrization process $X$ defined in (\ref{eq::SDE_parametrization_process_X}), after a deterministic time-change, is a Dyson-type diffusion. Then, we state the large deviation principle for the (time-changed) parametrization process $X^{x}=(X^{x}(t))_{t\ge 0}$ in the space $C_{x}([0,+\infty), D)$,  of continuous $D$-valued functions started at $x$, equipped with the topology of locally uniform convergence.

\subsection{Dyson-type diffusion}
\begin{definition}[\cite{abuzaid2024large}]\label{def::Dyson-type-potential}
 A function $U\in C^2(D,\mathbb{R})$ is said to be a Dyson-type potential if there exist constants $a>0$ and $B>0$ such that for all $x\in D$
\begin{equation}\label{eq::Dyson_type_potential}
    -B \le a\Delta U\le |\nabla U|^{2}+B.
\end{equation}
\end{definition}

\begin{proposition}
Suppose $\gamma\in C^{3}$. Then, the function 
\[
    V(x) = -\log\rho_{1,N}(x) =   - \sum_{i<j}\log|\gamma(x_{i})-\gamma(x_{j})|    
\]
is a Dyson-type potential, with some $B=B(\gamma, N)>0$ and any $a\le1$.
\end{proposition}
\begin{proof} 
In the course of the proof we set $\rho\equiv\rho_{1,N}$. The Laplacian of $V$ equals
\[
    \Delta V = - \frac{\Delta \rho}{\rho } + \left|\frac{\nabla \rho}{\rho}\right|^{2}.
\]
In terms of $\rho$ the condition (\ref{eq::Dyson_type_potential}) reads
\[
    -\frac{B}{a} + \frac{a-1}{a}\left|\frac{\nabla\rho}{\rho}\right|^{2}\le \frac{\Delta\rho}{\rho}\le \left|\frac{\nabla\rho}{\rho}\right|^{2} + \frac{B}{a}.
\]
First, we compute the square of the logarithmic derivative:
\[
    \left|\frac{\nabla\rho}{\rho}\right|^2 
    = \sum\limits_{i}^{N}\left(\sum\limits_{j:j\neq i}\Re\left\{\frac{\gamma'(x_{i})}{\gamma(x_{i})-\gamma(x_{j})}\right\}\right)^2.
\]
Then, the expression for $\Delta \rho/\rho$ can written down as 
\[
    \frac{\Delta\rho}{\rho} 
    = \left|\frac{\nabla\rho}{\rho}\right|^2
    -  \sum\limits_{i\neq j}\Re\left\{\left(\frac{\gamma'(x_{i})}{\gamma(x_{i})-\gamma(x_{j})}\right)^{2}\right\}
    + \sum\limits_{i\neq j}\Re\left\{\frac{\gamma''(x_{i})}{\gamma(x_{i})-\gamma(x_{j})} \right\}.
\]
\paragraph{$1^{\circ}.$ Upper bound.} Let us define the following constants that depend of the curve:
\begin{align*}
    &C_{1}(\gamma) := \sup\limits_{x_{1}<x_{2}<x_{1}+l}\left|\Im\left\{\frac{\gamma'(x_{1})}{\gamma(x_{1})-\gamma(x_{2})}\right\}\right|,\\
    &C_{2}(\gamma) :=\sup\limits_{x_{1}<x_{2}<x_{1}+l}\left|\Re\left\{\frac{\gamma''(x_{1})-\gamma''(x_{2})}{\gamma(x_{1})-\gamma(x_{2})} \right\}\right|.
\end{align*}
Then, $\Delta\rho/\rho$ has the upper bound 
\begin{align*}
    \frac{\Delta\rho}{\rho} 
    &\le \left|\frac{\nabla\rho}{\rho}\right|^2
    +  \sum\limits_{i\neq j}\Im\left\{\frac{\gamma'(x_{i})}{\gamma(x_{i})-\gamma(x_{j})}\right\}^{2}
    + \sum\limits_{i\neq j}\Re\left\{\frac{\gamma''(x_{i})}{\gamma(x_{i})-\gamma(x_{j})} \right\}\\
    &\le \left|\frac{\nabla\rho}{\rho}\right|^2 + N^{2}(C_{1}^{2} +C_{2}).
\end{align*}
\paragraph{$2^{\circ}.$ Lower bound.} We would like to determine the values of $a>0$ such that the two middle terms on the right hand side of 
\[
    \frac{\Delta\rho}{\rho} 
    =\frac{a-1}{a}\left|\frac{\nabla\rho}{\rho}\right|^2 
    + \frac{1}{a}\left|\frac{\nabla\rho}{\rho}\right|^2 
    - \sum\limits_{i\neq j}\Re\left\{\left(\frac{\gamma'(x_{i})}{\gamma(x_{i})-\gamma(x_{j})}\right)^{2}\right\} 
    +\sum\limits_{i\neq j}\Re\left\{\frac{\gamma''(x_{i})}{\gamma(x_{i})-\gamma(x_{j})} \right\}.
\]
are bounded from below, that is, for some $A>0$ 
\[
    \frac{1}{a}\left|\frac{\nabla\rho}{\rho}\right|^2 
    - \sum\limits_{i\neq j}\Re\left\{\frac{\gamma'(x_{i})}{\gamma(x_{i})-\gamma(x_{j})}\right\}^{2} \ge -\frac{A}{a},
\]
which yields the condition 
\[
    \frac{1}{a}\ge \sup_{x\in D} \frac{\sum\limits_{i\neq j}\Re\left\{\frac{\gamma'(x_{i})}{\gamma(x_{i})-\gamma(x_{j})}\right\}^2-\sum\limits_{i\neq j}\Im\left\{\frac{\gamma'(x_{i})}{\gamma(x_{i})-\gamma(x_{j})}\right\}^2}{\sum\limits_{i\neq j}\Re\left\{\frac{\gamma'(x_{i})}{\gamma(x_{i})-\gamma(x_{j})}\right\}^2 + \sum\limits_{i\neq j\neq k}\Re\left\{\frac{\gamma'(x_{i})}{\gamma(x_{i})-\gamma(x_{j})}\right\}\Re\left\{\frac{\gamma'(x_{i})}{\gamma(x_{i})-\gamma(x_{k})}\right\}  + A}.
\]
Define the following constant that depends on the curve:
\begin{align*}
    C_{3}(\gamma) 
    :=\sup\limits_{x_{1}<x_{2}<x_{3}<x_{1}+l}\Bigg|
    &\Re\left\{\frac{\gamma'(x_{1})}{\gamma(x_{1})-\gamma(x_{2})}\right\}\Re\left\{\frac{\gamma'(x_{1})}{\gamma(x_{1})-\gamma(x_{3})}\right\}\\
    &+\Re\left\{\frac{\gamma'(x_{2})}{\gamma(x_{2})-\gamma(x_{1})}\right\}\Re\left\{\frac{\gamma'(x_{2})}{\gamma(x_{2})-\gamma(x_{3})}\right\}\\
    &+\Re\left\{\frac{\gamma'(x_{3})}{\gamma(x_{3})-\gamma(x_{1})}\right\}\Re\left\{\frac{\gamma'(x_{3})}{\gamma(x_{3})-\gamma(x_{2})}\right\}\Bigg|.
\end{align*}
Taking $A=N^{3}C_{3}$ we deduce that $a>0$ is allowed to be $a\le 1$ and, consequently, the following inequality holds:
\[
    \frac{\Delta\rho}{\rho} 
    \ge \frac{a-1}{a}\left|\frac{\nabla\rho}{\rho}\right|^2
    - N^{3}C_{3} - N^{2}C_{2}.
\]
$3^{\circ}.$ For $\gamma\in C^{3}$, all the constants $C_{i}(\gamma)$, $i=1,2,3$, are finite, and the claim holds with $B = \max\{N^{2}(C_{1}^{2}+C_{2}),N^{3}C_{3}\}$ and any $0<a \le 1$.
\end{proof}
\subsection{Extension to the time interval \texorpdfstring{$[0,+\infty)$}{[0,+∞)}}
Let $C_{x}([0,+\infty),D)$ be the space of continuous $D$-valued functions started at $x\in D$, equipped with the topology of locally uniform convergence, that is, uniform convergence on $[0,T]$ for every $T>0$.
Define $I:C_{x}([0,+\infty),D)\to[0,+\infty]$ by
\[
    I(f) = \frac{1}{2}\int_{0}^{+\infty}\sum\limits_{i=1}^{N}\left|\dot f_{i}(t) + \partial_{i} V(f(t))\right|^2dt
\]
if $f$ is absolutely continuous and set to $+\infty$ otherwise. Define also a truncated $I_{T}:C_{x}([0,T],D)\to[0,+\infty]$ given by 
\[
    I_{T}(f) = \frac{1}{2}\int_{0}^{T}\sum\limits_{i=1}^{N}\left|\dot f_{i}(t) + \partial_{i} V(f(t))\right|^2dt
\]
if $f$ is absolutely continuous on $[0,T]$ and set to $+\infty$ otherwise.

\begin{lemma}\label{lem::good_rate_function}
The functional $I:C_{x}([0,+\infty),D)\to[0,+\infty]$ is a good rate function.
\end{lemma}
\begin{proof}
For $f\in C_{x}([0,+\infty), D)$, by the proof of \cite[Corollary 2.13]{abuzaid2024large},  the following inequality holds
\[
    I(f)\ge \mathcal{D}(f)-V(x),
\]
where $\mathcal{D}(f)$ is the Dirichlet energy of the function $f$ given by 
\[
    \mathcal{D}(f) = \sum\limits_{i=1}^{N}\frac{1}{2}\int_{0}^{+\infty}\dot f_{i}(t)^2dt
\]
if $f$ is absolutely continuous and set to $+\infty$ otherwise. 

Let $c>0$. The above inequality implies that the subset $\{I(f)\le c \} \subset C_{x}([0,+\infty), D)$ is contained in $\{\mathcal{D}(f)\le c + V(x)\}$. The latter is compact since the Dirichlet energy $\mathcal{D}:C_{x}([0,+\infty), D)\to [0,+\infty]$ is a good rate function, see \cite[Chapter ??]{dembo2009large}. Thus, any sequence $(f_{n})\subset \{I(f)\le c \}$ has a convergent subsequence $(f_{n_{k}})$ since it is also contained in the compact set $\{\mathcal{D}(f)\le c + V(x)\}$; denote by $f^*$ its subsequential limit. 

Since the truncated functional $I_{T}:C_{x}([0,T],D)\to[0,+\infty]$ is a good rate function, see \cite[Lemma 2.15]{abuzaid2024large}, it is lower semi-continuous. 
\[ 
    I(f^{*}) = \lim \limits_{T\to+\infty}I_{T}(f^{*}) \le  \lim \limits_{T\to+\infty} \lim\limits_{k\to+\infty} I_{T}(f_{n_{k}})\le c.
\]
Hence, $f^*\in \{I(f)\le c\}$. It follows that the level set $\{I(f)\le c\}$ is a compact subset of $C_{x}([0,+\infty), D)$ for every $c>0$, and therefore the functional $I$ is a good rate function.
\end{proof}

Define the probability measure on $C_{x}([0,+\infty),D)$ by
\[  
    \mathbb{Q}_{\kappa}(E) = \mathbb{P}[X^{x}\in E], \quad E\subset C_{x}([0,+\infty),D).
\]
For $T>0$, the family $(\mathbb{Q}_{\kappa}^{T})_{\kappa>0}$ of probability measures on $C_{x}([0,T],D)$, given by  $\mathbb{Q}_{\kappa}^{T}(E) = \mathbb{P}[(X^{x}(t))_{t\in [0,T]}\in E]$, satisfies the large deviation principle, as $\kappa\to 0+$, with a good rate function given by the functional $I_{T}:C_{x}([0,T],D)\to[0,+\infty]$, see \cite[Theorem 1.8]{abuzaid2024large}.
\begin{theorem}\label{thm::extended_LDP}
The family $(\mathbb{Q}_{\kappa})_{\kappa>0}$ of probability measures on $ C_{x}([0,+\infty),D)$, equipped with the topology of locally uniform convergence, satisfies the large deviation principle with a good rate function $I:C_{x}([0,+\infty),D)\to[0,+\infty]$.
\end{theorem}
\begin{proof}
For $T>0$, let $C_{x}([0,T],D)$ be the space of continuous $D$-valued functions started at $x\in D$, equipped with the topology of uniform convergence. For $T'\ge T$ let $\pi_{T',T}:C_{x}([0,T'],D)\to C_{x}([0,T],D)$ be the restriction map, that is, $\pi_{T',T}((f(t))_{t\in [0,T']}) = (f(t))_{t\in [0,T]}$. The map $\pi_{T',T}$ is continuous and surjective, $\pi_{T,T}$ is the identity, and $\pi_{T'',T} = \pi_{T',T}\circ\pi_{T'',T'}$ for $T\le T'\le T''$. Let $\pi_{T}:C_{x}([0,+\infty),D) \to C_{x}([0,T],D)$ be the restriction map. Then $\pi_{T} = \pi_{T',T}\circ\pi_{T'}$ for $T\le T'$. The topology of locally uniform convergence is the weakest topology when the maps $\pi_{T}$ are continuous. The collection $(C_{x}([0,T],D), \pi_{T', T})_{T\le T'}$ defines a projective system.

By \cite[Theorem 4]{de1997exponential}, the LDP of $(\mathbb{Q}_{\kappa}^{T})_{\kappa>0}$ extends to $(\mathbb{Q}_{\kappa})_{\kappa>0}$, if
\begin{enumerate}
\item For every $T>0$, the family $\mathbb{Q}_{\kappa}^{T} = \mathbb{Q}_{\kappa}\circ \pi_{T}^{-1}$ satisfies the LDP on $C_{x}([0,T],D)$ with the rate function $I_{T}:C_{x}([0,T],D)\to [0,+\infty]$.
\item There exists a functional $I: C_{x}([0,+\infty),D)\to [0,+\infty]$ such that it has compact level sets and satisfies for all $T>0$ and $f^{T}\in C_{x}([0,T],D)$ the identity 
\[
    I_{T}(f^{T}) = \inf\{I(f): f\in \pi_{T}^{-1}(f^{T})\}.
\]
\end{enumerate}
The first item is the content of \cite[Theorem 1.8]{abuzaid2024large}. The second item is given by the functional $I:C_{x}([0,+\infty),D)\to[0,+\infty]$:
\[
    I(f) = \frac{1}{2}\int_{0}^{+\infty}\sum\limits_{i=1}^{N}\left|\dot f_{i}(t) + \partial_{i} V(f(t))\right|^2dt
\]
if $f$ is absolutely continuous and set $+\infty$ otherwise. Lemma~\ref{lem::good_rate_function} shows that it has compact level sets. 

Now we show the required identity. Any $f^{T} \in C_{x}([0,T],D)$ can be extended to a function in $C_{x}([0,+\infty),D)$ by 
\[
    f(t) = f^{T}(t)\mathds{1}_{[0,T)}(t) + \hat{f}^{T}(t-T)\mathds{1}_{[T,+\infty)}(t),
\]
where $\hat{f}^{T}$ is a solution the gradient system 
\[
\begin{cases}
    \dot u(t) = -\nabla V(u(t)), \quad t>0,\\
    u(0) = f^{T}(T).
\end{cases}
\]
The local existence of a solution to the gradient system follows from the Peano existence theorem. The global existence follows from two facts: (1) $t\to V(g(t))$ is non-increasing since 
\[
    \frac{d}{dt}V(u(t)) 
    = \nabla V(u(t)) \cdot \dot u(t) 
    = -|\nabla V(u(t))|^2\le 0.
\]
$V(x)\to +\infty $ as $x\to\partial D$, hence a solution to the gradient system with initial value $u(0)=x\in D$ stays inside $D_{V(x)}=\{y\in D: V(y)\le V(x)\}$, in particular it never hits the boundary $\partial D$. (2) The gradient $\nabla V$ is bounded on $D_{V(x)}$, since it is continuous and $l$-periodic in the direction $\hat e =(1,\dots,1)^{T}$, which implies the linear growth bound 
\[
    |u(t)|
    \le |x|+\int_{0}^{t}|\nabla V(u(s))|ds
    \le |x| + (\sup_{y\in D_{V(x)}}|\nabla V(y)|) t.
\]

The function $f$, extended in such a way, satisfies $I(f) = I_{T}(f^{T})+ I(\hat{f}^{T})$, where the second term 
\[
    I(\hat{f}^{T})
    =\frac{1}{2}\int_{0}^{+\infty}\sum\limits_{i=1}^{N}\left|\partial_{t}\hat{f}^{T}_{i}(t) + \partial_{i} V(\hat{f}^{T}(t))\right|^2dt
\]
vanishes due to $\partial_{t} \hat{f}^{T} = -\nabla V (\hat{f}^{T})$. Thus, we obtain the identity 
\[
     I_{T}(f^{T}) = I(f) = \inf\{I(\tilde f): \tilde f \in \pi_{T}^{-1}(f^{T})\}.\qedhere
\]
\end{proof}
\noindent This completes the proof of Theorem~\ref{thm:LDP_parametrization_process_infinite_time}.
\section{Hydrodynamic limit}\label{sec::hydro}
We will now study a particular many-particle limit of the process.
Let $\boldsymbol{Z}=(Z_{1},\dots,Z_{N})$ be Dyson Brownian motion on $\Gamma$ according to Definition~\ref{def_Dyson_BM_contour}.  Consider the empirical measure of the particles on the curve
\[
    \mu_{t}^{(N)}(dz) = \frac{1}{N}\sum\limits_{i=1}^{N}\delta_{Z_{i}(t/N)}(dz),
\]
and denote 
\[
    \mu_{t}^{(N)}(f) = \int_\Gamma f(z)\mu_{t}^{(N)}(dz) = \frac{1}{N}\sum\limits_{i=1}^{N}f(Z_{i}(\tfrac{t}{N})).
\]
%Here $\delta_z$ is a Dirac mass at $z$.
\begin{theorem}[Tightness and hydrodynamic limit]\label{thm::hydro_limit}
Suppose $\gamma \in C^2$. Let $\mu_{0}$ be a probability measure on $\Gamma$ and suppose that the initial distribution of particles is chosen so that $\mu_{0}^{(N)}\to \mu_{0}$, weakly as $N\to+\infty$. Then, the family $\{(\mu_{t}^{(N)})_{t\ge 0}\}_{N\ge 1}$ is tight, and any subsequential limit $\mu= (\mu_{t})_{t\ge0}$ satisfies
\begin{align}\label{eqn::hydro_SDE_2}
    \mu_{t}(f) = \mu_0(f) + \frac{\beta}{4}\int_0^t\int_\Gamma\int_\Gamma\left(\partial_s f(z)\Re\frac{\tau(z)}{z-w}-\partial_s f(w)\Re\frac{\tau(w)}{z-w}\right)\mu_t(dz)\mu_t(dw)dt,
\end{align}
for all functions $f$ smooth in a neighborhood of $\Gamma$.
\end{theorem}
\begin{proof} The proof adapts the arguments in \cite[Theorem 4.1]{cepa1997diffusing}  and \cite[Theorem 3.6]{ISlimit} to our setting.
Let $\mathcal{P}(\Gamma)$ be the space of probability measures on $\Gamma$ equipped with the topology of weak convergence; moreover, $\mathcal{P}(\Gamma)$ is metrizable. Let $\mathcal{M}_T{}=C([0,T], \mathcal{P}(\Gamma))$ be the space of continuous measure-valued processes $(\mu_{t})_{t\in[0,T]}$ equipped with the topology of uniform convergence. In the course of the proof $\mu_{\bullet}^{(N)}$ refers to the family $(\mu_{t}^{(N)})_{t\in[0,T]}$. 

We show that the random family $\{\mu_{\bullet}^{(N)}\}_{N\ge 1}\subset \mathcal{M}_{T}$ is tight for every $T>0$.  This follows from the tightness of random variables $\left\{\mu^{(N)}_{\bullet}(f)\right\}_{N\ge 1}$ in $C([0,T],\mathbb R)$ for every real-valued $f\in C^2(\mathcal{N}(\Gamma))$.

Denote $E(\boldsymbol{z}) = \frac{1}{2}\log\varrho_{\beta, N}(\boldsymbol{z})$, where the density $\varrho_{\beta, N}$ is defined in \eqref{eq:density_z}. For any real-valued $f\in C^2(\mathcal{N}(\Gamma))$, it follows from It\^o's formula, applied to the function $f\circ\gamma$, that 
\[
    f(Z_{i}(t)) = f(Z_{i}(0))
    +\int_{0}^{t}\left(\frac{1}{2}\partial^{2}_{s}f(Z_{i}(s)) 
    + (\partial_{s_{i}}E)(\boldsymbol{Z}(s))\partial_{s}f(Z_{i}(s))\right)ds 
    + \int_{0}^{t}\partial_{s}f (Z_{i}(s))dB_{i}(s).
\]
In turn, by a change of variable, we have
\begin{align}
    \mu_t^{(N)}(f)=\mu^{(N)}_{0}(f)+ 
    \frac{1}{N^2}\sum_{i=1}^{N}\int_{0}^{t}\left(\frac{1}{2}\partial_s^2 f(Z_i(\tfrac{s}{N}))
    +(\partial_{s_i}E)(\boldsymbol{Z}(\tfrac{s}{N}))\partial_s f(Z_i(\tfrac{s}{N}))\right)ds
    + M^{(N)}(t),\label{eqn:Ito_aux}
\end{align}
where 
\[
    M^{(N)}(t) = \frac{1}{N}\sum_{i=1}^{N} \int_{0}^{t}\partial_s f(Z_i(\tfrac{s}{N})) dB_{i}(\tfrac{s}{N})
\]
is a continuous martingale with 
\[
\mathbb{E}[|M^{(N)}(t)|^2]=\frac{1}{N^2} \sum_{i=1}^N\int_0^t \mathbb{E}[|\partial_sf(Z_i(\tfrac{s}{N}))|^2] ds.
\]

For the term with the second tangential derivative in \eqref{eqn:Ito_aux} we have the following deterministic bound
\begin{equation}\label{eqn:aux2}
\left|\frac{1}{N^2}\sum_{i=1}^{N}\int_{0}^{t}\frac{1}{2}\partial_s^2 f(Z_i(\tfrac{s}{N}))ds\right|\le \frac{T}{2N}\max_{z\in\Gamma}\left|\partial_s^2 f(z)\right|.
\end{equation}
The term with the interaction potential $E$ in \eqref{eqn:Ito_aux}  can be expressed as 
\begin{align*}
    &\frac{1}{N^{2}}\sum\limits_{i=1}^{N}(\partial_{s_i}E)(\boldsymbol{Z})\partial_s f(Z_i) \\
    &= \frac{\beta}{2N^{2}}\sum_{i=1}^N \sum_{j \neq i}\partial_{s}f(Z_{i})\Re\left\{\frac{\tau(Z_{i})}{Z_{i}-Z_{j}}\right\}\\
    &=\frac{\beta}{4N^{2}}\sum_{i=1}^N \sum_{j \neq i}\partial_{s}f(Z_{i})\left(\partial_{s}f(Z_{i})\Re\left\{\frac{\tau(Z_{i})}{Z_{i}-Z_{j}}\right\}-\partial_{s}f(Z_{j})\Re\left\{\frac{\tau(Z_{j})}{Z_{i}-Z_{j}}\right\}\right)\\
    &= \frac{\beta}{4} \iint_{(\Gamma\times\Gamma ) \smallsetminus \{z=w\}}\left(\partial_{s}f(z)\Re\left\{\frac{\tau(z)}{z-w}\right\}-\partial_{s}f(w)\Re\left\{\frac{\tau(w)}{z-w}\right\}\right)\mu^{(N)}_{\cdot}(dz)\mu^{(N)}_{\cdot}(dw).
\end{align*}
For $z,w\in \Gamma$ and $w\to z$ we have 
\begin{align}\label{eq::diagonal_first_line}
    &\lim_{w\to z}\left(\partial_{s}f(z)\Re\frac{\tau(z)}{z-w} - \partial_{s}f(w)\Re\frac{\tau(w)}{z-w}\right)\\
    &=\lim_{w\to z}\Re\left\{\tau(z)\frac{\partial_{s}f(z)-\partial_{s}f(w)}{z-w}\right\}
    +\lim_{w\to z}\Re\left\{\frac{\tau(z)-\tau(w)}{z-w}\partial_{s}f(w)\right\} \\
    &= \Re\left\{\tau(z)\partial_{z}\partial_{s}f(z)\right\} + \Re\left\{-ik(z)\partial_{s}f(z)\right\}\\
    &= \partial_{s}^{2}f(z),\label{eq::diagonal_last_line}
\end{align}
where $k(z)$ is the curvature of the curve at the point $z\in\Gamma$.
%By continuity, for $z\in\Gamma$ and $x_n\to z$ along $\Gamma$, we have
%\[\lim_{n\to\infty}\left(\partial_sf(z)\Re\frac{\tau(z)}{z-x_n}-\partial_s f(x_n)\Re\frac{\tau(x_n)}{z-x_n}\right)=\Re{\tau(z)}\partial_s(\tau \Re\partial_s f)(z)+i\Re{\tau(z)}\partial_s(\tau \Im\partial_s f)(z).\]
Thus, we obtain
\begin{align}\label{eqn:aux3}
    &\frac{1}{N^2}\sum_{i=1}^{N}(\partial_{s_i}E)(\boldsymbol{Z})\partial_s f(Z_i)\\
    &=\frac{\beta}{4}\int_\Gamma\int_\Gamma \left(\partial_s f(z)\Re\frac{\tau(z)}{z-w}-\partial_sf(w)\Re\frac{\tau(w)}{z-w}\right)\mu_{\cdot}^{(N)}(dz)\mu_{\cdot}^{(N)}(dw)\\
    &\hspace{9cm}- \frac{\beta}{4N}\int_{\Gamma}\partial_{s}^{2}f(z)\mu^{(N)}_{\cdot}(dz)\\
    &\le \frac{\beta}{4}\sup_{(z,w)\in \Gamma\times\Gamma}\left|\partial_s f(z)\Re\frac{\tau(z)}{z-w}-\partial_sf(w)\Re\frac{\tau(w)}{z-w}\right|
    + \frac{\beta}{4N} \max_{z\in\Gamma}|\partial_{s}^{2}f(z)|.
\end{align}
%\begin{align}\label{eqn:aux3}
%&\frac{1}{N^2}\sum_{i=1}^{N}(\partial_{s_i}E)(\boldsymbol{Z})\partial_s f(Z_i)\\
%&\begin{aligned}
%=&\frac{\beta}{4}\int_\Gamma\int_\Gamma \left(\partial_s f(x)\Re\frac{\tau(x)}{x-y}-\partial_sf(y)\Re\frac{\tau(y)}{x-y}\right)\mu_t^{(N)}(dx)\mu_t^{(N)}(dy)\\
%&-\beta\frac{1}{N}\int_\Gamma\left(\Re{\tau(z)}\partial_s(\tau \Re\partial_s f)(z)+i\Re{\tau(z)}\partial_s(\tau \Im\partial_s f)(z)\right)\mu_t^{(N)}(dz).
%\end{aligned}
%\end{align}
%The last term satisfies the deterministic bound
%\begin{equation}\label{eqn:aux4}
%    \left|\frac{\beta}{4N}\int_{\Gamma}\partial_{s}^{2}f(z)\mu^{(N)}_{t}(dz)\right|
%    \le \frac{\beta}{4N} \max_{z\in\Gamma}|\partial_{s}^{2}f(z)|.
%\end{equation}
%\begin{align}\label{eqn:aux4}
%&\left|\frac{1}{N}\int_\Gamma\left(\Re{\tau(z)}\partial_s(\tau %\Re\partial_s f)(z)+i\Re{\tau(z)}\partial_s(\tau \Im\partial_s f)(z)\right)\mu_t^{(N)}(dz)\right|\notag\\
%&\le\frac{1}{N}\left(\max_{z\in\Gamma}|\Re{\tau(z)}\partial_s(\tau \Re\partial_s f)(z)|+\max_{z\in\Gamma}|\Re{\tau(z)}\partial_s(\tau \Im\partial_s f)(z)|\right).
%\end{align}
Because $f\in C^{2}$ and \eqref{eq::diagonal_first_line}-\eqref{eq::diagonal_last_line}, the following constants are finite: 
\begin{align*}
    &C_{1}(\Gamma,f):=\sup_{(z,w)\in \Gamma\times\Gamma}\left|\partial_s f(z)\Re\frac{\tau(z)}{z-w}-\partial_sf(w)\Re\frac{\tau(w)}{z-w}\right|,\\
    &C_{2}(\Gamma, f) := \max_{z\in\Gamma}|\partial_{s}^{2}f(z)|. 
\end{align*}
Combining~\eqref{eqn:Ito_aux}--\eqref{eqn:aux3} together, we obtain 
\begin{align*}
    |\mu^{(N)}_{t}(f)-\mu^{(N)}_{s}(f)| 
    \le (t-s)\frac{\beta C_{1}(\Gamma,f)}{4}+   \frac{T(1+\beta/2)C_{2}(\Gamma,f)}{2N} + 2\max_{0\le t\le T}|M^{(N)}(t)|,
\end{align*}
and
\[
\begin{split}
    &\mathbb{P}\left[\sup_{|t-s|<\delta}|\mu^{(N)}_{t}(f)-\mu^{(N)}_{s}(f)| >\eta\right]\\
    &\le \mathbb{P}\left[\max_{0\le t\le T}|M^{(N)}(t)| >\frac{\eta}{2}-\frac{\delta\beta C_{1}(\Gamma, f)}{8}- \frac{T(1+\beta/2)C_{2}(\Gamma,f)}{4N}\right].
\end{split}
\]
For every $\varepsilon>0$, by Doob's martingale inequality, we can choose $N_{0}=N_{0}(\varepsilon)$ large enough such that for all $N\ge N_{0}$ 
\begin{align}\label{eqn:aux1}
    \mathbb P\left[\max_{0\le t\le T}\left|M^{(N)}(t)\right|\ge \varepsilon\right]
    &\le \frac{1}{\varepsilon^2}\mathbb E\left[\left|M^{(N)}(T)\right|^2\right]\\
    &= \frac{1}{\varepsilon^2 N^{2}} \sum_{i=1}^{N}\mathbb{E}\left[\int_{0}^{T}|\partial_{s}f(Z_{i}(\tfrac{s}{N}))|^{2}ds\right]\\
    &\le \frac{T \max_{z\in\Gamma}|\partial_{s}f(z)|^{2}}{\varepsilon^2 N}\\
    &\le \varepsilon.
\end{align}
It follows that for every $\varepsilon,\eta>0$ there exist $\delta>0$ and $N_{0}\in\mathbb{N}$ such that 
\[
    \mathbb{P}\left[\sup_{|t-s|<\delta}|\mu^{(N)}_{t}(f)-\mu^{(N)}_{s}(f)| >\eta\right]
    \le\varepsilon\qquad \forall N\ge N_{0}.
\]
By the Arzelà-Ascoli theorem, see \cite[Theorem 7.3]{billingsley2013convergence}, the sequence $\{\mu_{\bullet}^{(N)}(f)\}_{N\ge 1}$ of random processes in $C([0,T],\mathbb{R})$ is tight, and therefore the sequence $\{\mu_{\bullet}^{(N)}\}_{N\ge 1}$ in $\mathcal{M}_{T}$ is tight for any $T>0$.

Moreover, the proof shows that 
\[
\begin{split}
    \mu_{t}^{(N)}(f)
    =&\mu_{0}^{(N)}(f)\\
    &+\int_{0}^{t}\frac{\beta}{4}\int_\Gamma\int_\Gamma\left(\partial_s f(z)\Re\frac{\tau(z)}{z-w}-\partial_s f(w)\Re\frac{\tau(w)}{z-w}\right)\mu^{(N)}_s(dz)\mu^{(N)}_s(dw)ds \\
    &+ R^{(N)}(t),
\end{split}
\]
where $\sup_{t\in[0,T]}|R^{(N)}(t)|\to0$, as $N\to+\infty$, almost surely for any $T>0$. Consequently, any subsequential limit $\mu_{\bullet}$ is deterministic (since the initial condition is deterministic) and  satisfies \eqref{eqn::hydro_SDE_2}.
%which uniquely determines the subsequential limit.
\end{proof}

\begin{remark}
Suppose $\gamma\in C^{\infty}$. Assume that for each $N$, $\mu_{0}^{(N)}$ is the empirical measure corresponding to the initial configuration distributed according to the stationary law \eqref{eq:density_z} for Dyson Brownian motion on $\Gamma$. Then as $N\to+\infty$, $\mu_{0}^{(N)}$ converges weakly towards the electrostatic equilibrium measure $\nu_{\textrm{eq}}$ on $\Gamma$, that is, the harmonic measure on $\Gamma$ seen from $\infty$. Indeed, this follows from Corollary~1.4 of \cite{Courteaut-Johansson-2025}. By Theorem~\ref{thm::hydro_limit}, if $f$ is a smooth function on a neighborhood of $\Gamma$, we therefore have
\begin{equation}\label{eqn::hydro_ODE}
\int_\Gamma\int_\Gamma\left(\partial_s f(x)\Re\frac{\tau(x)}{x-y}-\partial_s f(y)\Re\frac{\tau(y)}{x-y}\right)\nu_{\textrm{eq}}(dz)\nu_{\textrm{eq}}(dw)=0.
\end{equation}
Hence, the electrostatic equilibrium measure $\nu_{\textrm{eq}}$ is a stationary solution to \eqref{eqn::hydro_ODE}. We expect that $\nu_{\textrm{eq}}$ is the unique stationary solution under suitable regularity assumptions. 
\end{remark}

\appendix
\section{Stationary FPK equation with gradient drift}
\begin{proposition}\label{prop:unique_station_FPK_smooth_case}
    Let $\rho\in C^{\infty}(D)$ and $\rho>0$ on $D$. If $u:\overline{D}\to\mathbb{R}$ is a non-negative solution to the stationary Fokker-Planck-Kolmogorov equation 
    \begin{align*}
        \text{div}\left(\frac{1}{2}\nabla u - u\frac{\nabla\rho}{2\rho}\right)=0
    \end{align*}
    in $D$, with reflecting boundary conditions, then there exist a constant $c>0$ such that
    $u = c \rho$.
\end{proposition}
\begin{proof}[Proof of Proposition~\ref{prop:unique_station_FPK_smooth_case}]
$1^{\circ}$ The assumptions on $\rho$ imply that $\frac{\nabla\rho}{2\rho} \in C^{\infty}(D)$. Let $u$ be a solution to $L^{*}u=0$ which can be rewritten as 
\[
\text{div}\left(\frac{1}{2}\nabla u - u\frac{\nabla\rho}{2\rho}\right)=0.
\]
Since the coefficients of this equation are smooth in $D$ Weyl's regularity theorem, see \cite[Chapter I]{bogachev2015FPK}, implies that a solution $u$ must be $C^{\infty}(D)$. In the weak form it is written as
\[
    \int_{ D}\nabla\phi\cdot \rho\nabla\left(\frac{u}{\rho}\right)dx =0 \quad \forall\varphi\in C_{0}^{\infty}( D).
\]
Let $f(t)=\frac{-1}{1+t}$, $\psi\in C^{\infty}_{0}( D)$ and $\xi\in C^{\infty}_{0}(\mathbb{R}^{N})$, then we can take $\varphi = f(\frac{u}{\rho})\psi\xi$ as a test function. 
\begin{align*}
    \int_{ D}\left|\nabla\frac{u}{\rho}\right|^{2}f'\left(\frac{u}{\rho}\right)\psi \rho dx 
    = -\int_{ D}  f\left(\frac{u}{\rho}\right) \rho\nabla\frac{u}{\rho}\cdot \xi\nabla\psi dx 
      -\int_{ D}  f\left(\frac{u}{\rho}\right) \rho\nabla\frac{u}{\rho}\cdot \psi\nabla\xi dx.
\end{align*}

$2^{\circ}.$ Let $\xi=\xi_{R}\in C^{\infty}_{0}(\mathbb{R}^{N})$ be given by $\xi_{R}(x)=\eta(\frac{x}{R})$ for some $\eta\in C^{\infty}_{0}(\mathbb{R}^{N})$, $0\le\eta\le1$ and $\eta(x)=1$ for $|x|<R$. Let $D_{\delta} = \{x\in D: \text{dist}(x,\partial D)\ge\delta\}$, $K_{\delta,R'} = D_{\delta}\cap \overline{B_{R'}}$, where $B_{R'}$ is a ball of radius $R'$ around the origin. $K_{\delta, R'}$ is a compact subset of $D$ with Lipschitz boundary. Let $\psi = \psi_{\varepsilon} = \mathds{1}_{K_{\delta, R'}}*\alpha_{\varepsilon}$ be a mollification of the indicator function.

The sequence $\psi_{\varepsilon}$ approximates the indicator function $\mathds{1}_{K_{\delta, R'}}$. The measure $\nabla \psi_{\varepsilon}dx$ converges weakly to the surface $-n d\sigma$ measure on $\partial K_{\delta, R'}$, where $n$ is the inward-pointing normal unit vector,  \cite[Chapter 5]{evans2025measure}:
\begin{align*}
    \lim\limits_{\varepsilon\to 0+}\int g\cdot  \nabla\psi_{\varepsilon}dx
    = -\int_{\partial K_{\delta, R'}} g\cdot n dS\quad\forall g\in C_{b}(U(\partial K_{\delta, R'})),
\end{align*}
where $U(\partial K_{\delta, R'})$ is an open neighborhood of $\partial K_{\delta, R'}$. The function $f\left(u/\rho\right)\rho\nabla\left(u/\rho\right)\xi_{R}$ is smooth on $D$, so it leads to 
\begin{align*}
    \lim\limits_{\varepsilon\to 0+}\int_{D}f\left(\frac{u}{\rho}\right) \rho\nabla\frac{u}{\rho}\cdot\xi\nabla\psi_{\varepsilon}dx 
    = -\int_{\partial K_{\delta, R'}}f\left(\frac{u}{\rho}\right)  \rho n\cdot\nabla\left(\frac{u}{\rho}\right) \xi_{R} d\sigma.
\end{align*}
and the equality 
\begin{align*}
    \int_{ K_{\delta, R'}}\left|\nabla\frac{u}{\rho}\right|^{2}f\left(\frac{u}{\rho}\right)^2\rho dx
    = \int_{\partial K_{\delta, R'}}f\left(\frac{u}{\rho}\right)n\cdot\rho\nabla\frac{u}{\rho} \xi_{R} d\sigma 
    -\int_{ K_{\delta, R'}}  f\left(\frac{u}{\rho}\right) \rho\nabla\frac{u}{\rho}\cdot \nabla\xi_{R} dx.
\end{align*}
The last term can be estimated by
\begin{align*}
    \left|\int_{ K_{\delta, R'}}  f\left(\frac{u}{\rho}\right) \rho\nabla\frac{u}{\rho}\cdot \nabla\xi_{R} dx\right| 
    \le \frac{||\eta||_{\infty}}{R}\left(\int_{K_{\delta, R'}} f\left(\frac{u}{\rho}\right)^{2}\left|\nabla\frac{u}{\rho}\right|^{2}\rho dx\right)^{1/2}\left(\int_{K_{\delta, R'}} \rho dx\right)^{1/2}.
\end{align*}
The integral $\int_{K_{\delta, R'}} \rho dx$ is order of $R'$. Taking $R' = cR$ for some $c>1$ we can find a constant $C>0$ such that 
\begin{align*}
    \int_{K_{\delta, R'}} f\left(\frac{u}{\rho}\right)^{2}\left|\nabla\frac{u}{\rho}\right|^{2}\rho dx
    \le \left|\int_{\partial K_{\delta, R'}}f\left(\frac{u}{\rho}\right)n\cdot\rho\nabla\frac{u}{\rho} \xi_{R} d\sigma \right| 
    + \frac{C}{\sqrt{R}}\left(\int_{K_{\delta, R'}} f\left(\frac{u}{\rho}\right)^{2}\left|\nabla\frac{u}{\rho}\right|^{2}\rho dx\right)^{1/2}.
\end{align*}
Hence,
\begin{align*}
    \int_{K_{\delta, R'}} f\left(\frac{u}{\rho}\right)^{2}\left|\nabla\frac{u}{\rho}\right|^{2}\rho dx
    \le \frac{1}{4}\left(\frac{C}{\sqrt{R}}+ 
    \sqrt{\frac{C^2}{R}+ 4  \left|\int_{\partial K_{\delta, R'}}f\left(\frac{u}{\rho}\right)n\cdot\rho\nabla\frac{u}{\rho}  d\sigma \right| }\right)^{2}.
\end{align*}
Taking the limit $R\to\infty$
\begin{align*}
    \int_{ D_{\delta}}f\left(\frac{u}{\rho}\right)^{2}\left|\nabla\frac{u}{\rho}\right|^{2}\rho dx
    \le \left|\int_{\partial D_{\delta}}f\left(\frac{u}{\rho}\right)n\cdot\rho\nabla\frac{u}{\rho}d\sigma\right|.
\end{align*}
$3^{\circ}.$ Reflecting boundary condition 
\begin{align*}
    \lim\limits_{K\uparrow D}\int_{\partial K}\varphi \left(n\cdot\rho\nabla\frac{u}{\rho}\right)d\sigma 
    =0 \quad\forall \varphi\in C^{2}(\overline{D})
\end{align*}
leads to 
\begin{align*}
    \int_{ D}\left|\nabla\frac{u}{\rho}\right|^{2}\frac{1}{(1+u/\rho)^2}\rho dx
    = 0.
\end{align*}
Thus, $u/\rho$ must be a constant on $D$.
\end{proof}

\printbibliography
\end{document}
\typeout{get arXiv to do 4 passes: Label(s) may have changed. Rerun}